\documentclass[11pt]{amsart}

\usepackage{amsmath,amssymb, graphics, amscd,latexsym}
\usepackage{amsmath,amssymb, graphicx, amscd,latexsym,here}
\usepackage{verbatim}
\usepackage{amssymb}
\usepackage{amsbsy}
\usepackage{amscd}
\usepackage{amsmath}
\usepackage{amsthm}
\usepackage[mathscr]{eucal}
\makeatletter
\newtheorem{Theorem}{Theorem}
\newtheorem{Lemma}[Theorem]{Lemma}
\newtheorem{Corollary}[Theorem]{Corollary}
\newtheorem{Proposition}[Theorem]{Proposition}

\newtheorem{Remark}[Theorem]{Remark}

\newtheorem{ Observation}[Theorem]{Observation}
\newtheorem{Assertion}[Theorem]{Assertion}

\newcommand{\eps}{\varepsilon}

\newcommand\vphi{\varphi}

\newcommand\al{\alpha}

\newcommand\si{\sigma}
\newcommand\be{\beta}
\newcommand\Si{\Sigma}
\newcommand\ga{\gamma}
\newcommand\Ga{\Gamma}
\newcommand\de{\delta}
\newcommand\De{\Delta}

\newcommand\cM{\mathcal  M}

\newcommand\cT{\mathcal  T}

\newcommand\cS{\mathcal  S}

\newcommand\cV{\mathcal  V}

\newcommand\cK{\mathcal  K}

\newcommand\BC{ {\mathbb C}}
\newcommand\BN{ {\mathbb  N}}

\newcommand\BZ{{\mathbb  Z}}
\newcommand\BR{ {\mathbb  R}}
\newcommand\BP{ {\mathbb  P}}

\newcommand\bfg{\mbox {\bf  g}}

\newcommand\bfu{\mbox {\bf  u}}

\newcommand\bfv{\mbox {\bf  v}}

\newcommand\bfx{\mbox {\bf  x}}

\newcommand\bfw{\mbox {\bf  w}}

\newcommand\bfz{\mbox {\bf  z}}

\newcommand\bfa{\mbox {\bf  a}}

\newcommand\bfb{\mbox {\bf  b}}

\newcommand\nl{\newline}

\newcommand\dist{\rm{dist}\/}

\newcommand\Int{\rm{Int}\/}

\newcommand\grad{\rm{grad}\/}

\newcommand\rank{\rm{rank}\/}

\newcommand\Cone{\rm{Cone}\/}

\newcommand\Vol{\rm{Vol}\/}

\newcommand\prosim{\overset{proj}{\sim}}

\newcommand\rdeg{{\rm{rdeg}\/}}
\newcommand\pdeg{{\rm{pdeg}\/}}

\newcommand\inv{^{-1}}

\def\mapright#1{\smash{\mathop{\longrightarrow}\limits^{{#1}}}}

\def\mapdown#1{\Big\downarrow\rlap{$\vcenter{\hbox{$#1$}}$}}

\def\mapup#1{\Big\uparrow\rlap{$\vcenter{\hbox{$#1$}}$}}

\def\inv{^{-1}}

\begin{document}
\title[Mixed functions of strongly polar weighted
homogeneous face type
]
{Mixed functions of strongly polar weighted
homogeneous face type
}

\author
[M. Oka ]
{Mutsuo Oka }
\address{Department of Mathematics,
{Tokyo  University of Science}
}
\email { oka@rs.kagu.tus.ac.jp}

\begin{abstract}
Let $f(\bfz,\bar\bfz)$ be a mixed polynomial with strongly
 non-degenerate face functions.
We consider a canonical toric modification
$\pi:\,X\to \BC^n$ and a 
polar modification $\pi_{\BR}:Y\to X$.
We will show that 
the toric modification 
resolves   topologically  the singularity of $V$ and the zeta  function of
 the Milnor  fibration of $f$ is described  by a formula of
 a  Varchenko type.

\end{abstract}
\maketitle

\maketitle

\section{Introduction}
Recall that 
a mixed polynomial $f(\bfz,\bar\bfz)$ with $n$ complex variables variables $\bfz=(z_1,\dots, z_n)\in
\BC^n$
is called 
 a {\it polar weighted  homogeneous polynomial }
if there exist a weight vector 
$P=(p_1,\dots, p_n)$ and positive integers $d_p$ 
such that 
\begin{eqnarray*}\begin{split}
&f(\rho\circ \bfz,\bar\rho\circ\bar \bfz)=\rho^{d_p}f(\bfz,\bar\bfz),\quad
\rho\circ \bfz=(\rho^{p_1}z_1,\dots, \rho^{p_n}z_n),\,\rho\in  \BC,\,|\rho|=1.
\end{split}\end{eqnarray*}
Similarly $f(\bfz,\bar\bfz)$ is called 
 a {\it radially  weighted  homogeneous  polynomial}
if there exist  a weight vector $Q=(q_1,\dots, q_n)$ and a positive
 integer  $d_r$ such  that 
\begin{eqnarray*}\begin{split}
&f(t\circ \bfz,t\circ
		  \bar\bfz)=t^{d_r}f(\bfz,\bar\bfz),
\quad t\circ \bfz=(t^{q_1}z_1,\dots, t^{q_n}z_n),\,t\in \BR^+\\
\end{split}\end{eqnarray*}
If $f$ is both radially and polar weighted homogeneous,
we have  an associated  $\BR^+\times S^1$-action on $\BC^n$ by
\[
 (t,\rho)\circ \bfz=(t^{q_1}\rho^{p_1}z_1,\dots,
 t^{q_n}\rho^{p_n}z_n),\quad (t,\rho)\in \BR^+\times S^1.
\]
The integers $d_r$ and $d_p$ are called the radial and the polar
degree respectively and we denote them as
$d_r=\rdeg_Qf$ and $d_p=\pdeg_Pf$.
Usually a polar weighted homogeneous polynomial is also assumed
to be radially weighted homogeneous {\cite{OkaPolar}}. 
We assume  this throughout   in this paper.

We say that $f(\bfz,\bar \bfz)$ is {\em  strongly polar weighted  homogeneous}
if  $p_j=q_j$ for $j=1,\dots, n$. Then the associated $\BR^+\times S^1$ 
action on $\BC^n$ reduces to  a $\BC^*$ action which is defined by
\[
 (\tau,\bfz)=(\tau,(z_1,\dots, z_n)\mapsto \tau\circ\bfz=(z_1
 \tau^{p_1},\dots, z_n\tau^{p_n}),\,\,\tau\in \BC^*.
\]
Furthermore
 $f$ is called {\em a strongly polar
 positive weighted homogeneous polynomial} if 
 $\pdeg_P f>0$.

A mixed function
$f(\bfz,\bar \bfz)$ is called {\it
of  strongly  polar positive weighted
homogeneous face type}
if for each face $\De$ of  dimension $n-1$, 
$f_{\De}(\bfz,\bar\bfz)$
is a   strongly polar positive weighted homogeneous polynomial.

The purpose of this paper is to generalize the result
 of Varchenko (\cite{Varchenko})
to non-degenerate mixed functions of strongly polar weighted
homogeneous face type (Thorem \ref{main-result}).

\section{Non-degeneracy and associated toric modification}
Throughout this paper, we use the same notations as in \cite{Okabook,OkaMix}, 
unless we state otherwise. 
We recall basic terminologies for the  toric modification.
\subsection{Non-degenerate  functions}
Let $f(\bfz)=\sum_{\nu}a_{\nu,\mu}\bfz^\nu{\bar\bfz}^{\mu}$ be a
convenient  
mixed analytic function.
The Newton polyhedron
$\Ga_+(f)$ is defined by the convex hull of 
the union 
$\cup_{\nu}\{\nu+\mu+\BR_+^n\,|\, a_{\nu,\mu}\ne 0\}$.
The Newton boundary $\Ga(f)$ is the union of the compact faces of $\Ga_+(f)$.
If $f$ is a holomorphic function germ,  $a_{\nu,\mu}=0$
unless $\mu=(0,\dots,0)$ and the Newton boundary $\Ga(f)$ coincides with
 the usual one.
For a positive weight vector $P={}^t(p_1,\dots, p_n)$, we associate a
linear function $\ell_P$ on $\Ga(f)$ by
$\ell_P(\tau)=\tau_1p_1+\cdots+\tau_n p_n$ for $\tau\in \Ga(f)$.
It takes a minimum value which we denote by $  d(P,f)$
or $d(P)$ if $f$ is fixed. Let $\De(P)$ be the face where $\ell_P$ takes
the minimal value and put $f_P:=\sum_{\nu+\mu\in
\De(P)}a_{\nu,\mu}\bfz^\nu{\bar\bfz}^\mu$ and we call $f_P$ {\em the face
function of $f$ with respect to $P$}. 

Recall that 
$f$ is {\em non-degenerate for
$P$}   (respectively {\em strongly non-degenerate for $P$}) if
the polynomial mapping
$f_P:\BC^{*n}\to \BC$ has no critical point on $f_P\inv(0)$
(resp. on $\BC^{*n}$).
In the case that  $f_P$ is a polar weighted homogeneous polynomial,
two notions  coincide (\cite{OkaMix}). In particular,  two notions
for non-degeneracy coincide for holomorphic functions.
 
Consider a mixed monomial $M=\bfz^\nu{\bar\bfz}^\mu$.
The radial degree $\rdeg_P(M)$ and polar degree
$\pdeg_P(M)$ with respect to $P$ is defined by
\[
 \rdeg_P(M)=\sum_{i=1}^n p_i(\nu_i+\mu_i),\quad
\pdeg_P(M)=\sum_{i=1}^n p_i(\nu_i-\mu_i)
\]
Note that the face function  $f_P$ is a radially weighted 
homogeneous polynomial
of degree $d(P)$ by the definition.

Consider the space of positive weight vectors $N^+$.
Recall that an equivalent relation $\sim$ on $N^+$ is  defined by
\[\text{for}~~P,Q\in N^+,\,\,
 P\sim Q\iff \De(P)=\De(Q).
\]
This defines a conical subdivision of $N^+$ which is called {\em the dual
Newton diagram for $f$ } and we denote it by $\Ga^*(f)$.
\subsection{Admissible subdivision and an admissible toric modification}
We recall the admissible toric modification for  beginner's convenience.
We first take a
regular simplicial subdivision $\Si^*$ of the dual Newton diagram $\Ga^*(f)$.
Such a regular fan is called an admissible regular fan.
See \cite{Okabook} for the definition.
The primitive generators
of one dimensional cones in $\Si^*$ are called vertices.
Namely a vertex has a unique expression as a primitive integral vector
$P={}^t(p_1,\dots, p_n)$ with $\gcd(p_1,\dots, p_n)=1$.
$P$ is {\em strictly positive} if $p_j>0$ for any $j$.
Let $\cV$ be the vertices of $\Si^*$
and let $\cV^+\subset \cV$ be the vertices  which are strictly positive.
(We denote the strict positivity by $P\gg 0$.)
To each $n$-dimensional simplicial cone $\tau$ of $\Si^*$, we associate
a unimodular matrix, which we  denote it by $\tau$ by  an abuse of notation. Thus if $P_1,\dots, P_n$
are primitive vertices of $\tau$, we also identify
 $\tau$ with   the  unimodular matrix $(P_1,\dots, P_n)\in {\rm SL}(n;\BZ)$.
On the other hand,  as a cone,
$\tau=\{\sum_{i=1}^k a_iP_i\,|\, a_i\ge 0, i=1,\dots,n\}$.
We say that $\Si^*$ is {\it convenient} if 
the vertices of $\Si^*$ are strictly positive except the
obvious elementary ones
$E_j={}^t(0,\dots,1,\dots 0)$ (1 is at $j$-th coordinates), $j=1,\dots,
n$.
We assume that $f$ is convenient and thus 
we   assume also that $\Si^*$ is convenient hereafter.

 We denote by $\cK$ (respectively by  $\cK_s$ ) the set of 
simplicies of $\Si^*$ (resp. $s$-simplices of $\Si^*$).
Note that an $s$-simplex corresponds to  an $(s+1)$-dimensional cone.
For each  $\tau=(P_1,\dots, P_n)\in\cK_{n-1}$,
we associate affine space $\BC_\tau^n$ with the toric coordinates
$\bfu_\tau=(u_{\tau 1},\dots, u_{\tau n})$ and a toric morphism
$ \pi_\tau:\, \BC_\tau^n\to \BC^n$ with
$\bfz=\pi_\tau(\bfu_\tau)$, $z_j=u_{\tau 1}^{p_{j1}}\cdots u_{\tau
n}^{p_{jn}}$
for $j=1,\dots,n$.
where $\BC^n$ is the base space and $\bfz=(z_1,\dots, z_n)$ is the fixed 
coordinates.
Let $X$ be the quotient space of the disjoint union
$\amalg_\si \BC_\si^n$ by the canonical identification
$\bfu_\tau\sim \bfu_\si$ iff $\bfu_\tau=\pi_{\tau\inv \si}(\bfu_\si)$
where 
$\pi_{\tau\inv \si}$ is well defined on $\bfu_\si$.
The quotient space is  a complex manifold of dimension $n$ and we have a
canonical projection
 $\pi: X\to \BC^n$
  which is called {\em  the associated toric modification}. 
Recall that  $\pi$
gives a birational morphism such that
$\pi:X\setminus \pi\inv(\bf 0)\to \BC^n\setminus\{{\bf 0}\}$
is an isomorphism, as we have assumed that $\Si^*$ is convenient.
Here ${\bf 0}$ is the origin of $\BC^n$.
It also gives  a good resolution of the function germ $f$
at the origin if $f(\bfz)$ is a non-degenerate holomorphic
function germ.  However for a mixed non-degenerate germ, $\pi$  does not give a good
resolution
in general (\cite{OkaMix}).

\subsection{Configuration of the exceptional divisors}
We recall the configuration  of exceptional divisors of $\pi:X\to \BC^n$. 
 For further detail, see \cite{Okabook}.
For each
vertex $P\in \cV^+$ of $\Si^*$, there corresponds an exceptional divisor
$\hat E(P)$. 
The restriction $\pi:X\setminus\pi\inv({\bf 0})\to \BC^n\setminus\{{\bf 0}\}$
is biholomorphic and the exceptional fiber $\pi\inv(\bf0)$ is described  as:
\[
 \pi\inv({\bf 0})=\cup_{P\in \cV^+}\hat E(P).
\]
Note that $\cV\setminus \cV^+=\{E_j;j=1,\dots, n\}$ and
$\hat E(E_j)$ is not compact and $\pi|_{\hat E(E_j)}: \hat E(E_j)\to \{z_j= 0\}$
is  biholomorphic.
 Let $ \widetilde V$ be the strict transform of $V$ to $X$.
Recall that $E(P):=\hat E(P)\cap  \widetilde V$ is non-empty if and only if 
$\dim\,\De(P;f)\ge 1$.
\subsubsection{Stratification.}
We define  {\em the toric stratification } and {\em
the Milnor strafitication} of the exceptional fiber $\pi\inv({\bf 0})$. 
For each simplex $\tau=(P_1,\dots, P_k)$ of $\Si^*$, we define
\begin{eqnarray*}
& \hat E( \tau)^*=\cap_{i=1}^k\hat E(P_i)\setminus\cup_{Q\in\cV,Q\notin
   \tau}\hat E(Q),\\
& \widetilde V(\tau)^*=\hat E(\tau)^*\cap  \widetilde V,\,\,
 \widetilde E(\tau)^*=\hat E(\tau)^*\setminus \widetilde V(\tau).
\end{eqnarray*}
In the case of $\tau=(P)$, we simply write $\hat E(P),\, \widetilde V (P)^*$ and $ \widetilde E(P)^*$.
Then we consider two caninical stratifications of $\pi\inv({\bf 0})$:
\begin{eqnarray}
&\text{Toric stratification}:\quad \cT:=\{\hat E(\tau)^*\,|\, \tau\cap \cV^+\ne \emptyset\},\,\\
&\text{Milnor stratification}: \cM:=\{ \widetilde E(\tau)^*, \widetilde V(\tau)^*\,|\, \tau\cap \cV^+\ne \emptyset\}.
\end{eqnarray}
Here $\tau\cap \cV^+\ne \emptyset$ implies $\hat E(\tau)\subset \pi\inv(\bf0)$.
We call $\tau$ {\em the support simplex} of $ \widetilde E(\tau), \widetilde
V(\tau)$.
 If $\tau$ is a subsimplex of $\sigma$,  we denote it as   $\tau\prec\sigma$. The basic properties are
\begin{Proposition}
\begin{enumerate}
\item $\hat E(P)\cap \hat E(Q)\ne \emptyset$ if and only if 
$(P,Q)$ is a simplex of $\Si^*$.
\item Let $\tau=(P_1,\dots, P_k)$ is a $k$-simplex and let
      $\si=(P_1,\dots, P_n)$ and 
$\si'=(P_1,\dots, P_k,Q_{k+1},\dots, Q_n)$ be $(n-1)$-simplices for
		    which
$\tau\prec\sigma$ and $\tau\prec\si'$.
Put
\begin{eqnarray*}
& \hat E(\tau)_\si^*:=\{\bfu_\si\in \BC_\si^n\,|\,
u_{\si,i}=0,i\le k,\,
u_{\si,j}\ne 0,\,   j\ge k+1\}\\
&\hat E(\tau)_{\si'}^*:=\{\bfu_{\si'}\in \BC_{\si'}^n\,|\,
u_{\si',i}=0,i\le k,\,
u_{\si',j}\ne 0,\,j\ge k+1\}.
\end{eqnarray*}
Then we have $ \hat E(\tau)_\si^*=\hat E(\tau)_{\si'}^*$.
In particular, 
\nl
$\hat E(\tau)^*=\hat E(\tau)_\si^*\cong \BC^{*(n-k)}$.
\item $\amalg_{\tau,P\in \tau}\hat E(\tau)^*$ is a toric stratification of $\hat E(P)$.
\end{enumerate}
\end{Proposition}
\begin{proof}
As a unimodular matrix, $\si\inv\si'$ takes the following form
\[
 \si\inv\si'=\left(
\begin{matrix}
I_k&B\\ O& C
\end{matrix}\right)
\] where $O$ is $(n-k)\times k$ zero matri
and $I_k$ is the $k\times k$ identity matrix.
From this expression, it is clear that the restriction
of  the morphism
$\pi_{\si\inv\si'}:\BC_{\si'}^{*n}\to \BC_{\si}^{*n}$
gives the isomorphism 
$\pi_C:\hat E(\tau)_{\si'}^*\to \hat E(\tau)_{\si}^*$
where $\pi_C$ is the toric morphism associated with the unimodular
 matrix $C$. The other assertion is obvious. See \cite{Okabook} for
 further detail.
\end{proof}
\subsection{Milnor fibration}
Let $f$ be a strongly non-degenerate function which is either holomorphic or mixed  analytic.
We consider the Minor fibration by the second description:
$ f: E(\eps,\de)^*\to D_\de^*$
where 
\[\begin{split}
&E(\eps,\de)^*=B_\eps^{2n}\cap f\inv(D_\de^*)\\
& B_\eps^{2n}=\{\bfz\in \BC^n\,|\, \|\bfz\|\le \eps\},\,
D_\de^*:=\{\rho\in \BC\,|\,0< |\rho|\le \de\}.
\end{split}
\]
The Milnor fiber is given by $F_{\eta,\eps}:=f\inv(\eta)\cap B_\eps^{2n}$
with $0\ne |\eta|\le \de$. Note that  as long as $\eps$ is smaller than the stable radius
$\eps_0$ and $\de\ll \eps$, the fibering structure does not depend on the choice of $\eps$ and $\de$.

Let $\pi: X\to \BC^n$ be the associated toric modification. 
The restriction $\pi:X\setminus\pi\inv({\bf 0})\to \BC^n\setminus\{{\bf 0}\}$
is biholomorphic.
Then Milnor fibration
 can be replaced by
 $\pi^* f=f\circ \pi:\hat E(\eps,\de)^*\to D_\de^*$
where
\begin{eqnarray}
&\hat E(\eps,\de)^*=\{x\in X\,|\,
0< |f(\pi(x))|\le \de\}\cap  \widetilde B_\eps\notag \\
& \widetilde B_\eps=\{x\,|\, \|\pi(x)\|\le \eps\}.\notag
\end{eqnarray}
Note that $ \widetilde B_\eps$ can be understood as an $\eps$-neighborhood of 
$\pi\inv({\bf 0})$.
 Let $ \widetilde V$ be the stric transform of $V$ to $X$. The  above setting 
 is common for holomorphic functions and mixed functions.
 
 \section{A theorem  of Varchenko}
 We first recall the result of Varchenko for a non-degenerate convenient
holomorphic function
$f(\bfz)$.
Consider a germ of hypersrface $V=f\inv(0)$.
For  $I\subset \{1,\dots, n\}$, let $f^I$ be the restriction of 
$f$ on the coordinate subspace $\BC^I$
where
\[
 \BC^I=\{\bfz\,|\,z_j=0,\,j\notin I\},\,\,
\BC^{*I}=\{\bfz\,|\,z_j=0\iff j\notin I\}.
\]
Let  $\cS_I$ be the set of primitive weight vectors
$P={}^t(p_i)_{i\in I}$  of the variables $\{z_i\,|\, i\in I\}$ such that 
$ p_i>0$  for all $ i\in I$ and $\dim\,\De(P,f^I)=|I|-1$.
$P\in \cS_I$ can be considered to be a weight vector of $\bfz$ putting $p_j=0,\,j\notin I$.
Then the result of Varchenko (\cite{Varchenko}, see also \cite{ Okabook})
can be stated as follows.
\begin{Theorem}\label{Varchenko}
The zeta function of the Milnor fibration of $f$ is given by the formula
\[
 \zeta(t)=\prod_I
\zeta_I(t),\quad \zeta_I(t)=\prod_{P\in
\cS_I}(1-t^{  d(P,f^I)})^{-\chi(P)/  d(P,f^I)}.\]
The term
$\chi(P)$ is the Euler-Poincar\'e characteristic of 
toric Milnor fiber $F(P)^*$  where 
\[
 F(P)^*:=\{\bfz^I\in \BC^{*I}\,|\, f_P^I(\bfz^I)=1\}.
\]
and it is an combinatorial invariant which satisfies the equality:
\begin{eqnarray}\label{volume}
 \chi(P)=(-1)^{|I|-1}|I|!\Vol_{|I|} \Cone(\De(P;f^I),{\bf 0}).
\end{eqnarray}
\end{Theorem}
\section{Revisit to the proof}

 For the proof of Theorem \ref{Varchenko}, we use an admissible toric
 modification as in the proof in \cite{Okabook}.
We will generalize this theorem for a convenient non-degenerate mixed function
of  strongly polar weighted homogeneous face type in the next section.
For this purpose, we give a detailed description of the proof  so that it can be used for a 
mixed function of strongly polar weighted homogeneous face type without
 any 
essential change.
%
Let $\Si^*$ be an admissible regular, convenient subdivision and let 
$\pi:X\to \BC^n$ be the associated toric modification.
%

\subsection{Compatibility of the charts}
Let $\tau=(P_1,\dots, P_k)\in \cK_{k-1}$  and 
suppose that we have two coordinate charts
$\si$ and $\si'$ such that $\tau\prec \si,\,\si'$ and 
  $\tau=\si\cap \si'$. Put
$\si=(P_1,\dots, P_n)$
and $\si'=(P_1,\dots, P_k, Q_{k+1},\dots, Q_n)$.
We also assume that $\hat E(\tau)^*\in \cT$. This implies
$\pi(\hat E(\tau))=\{\bf0\}$.
Then we have
\begin{Proposition}\label{sigma}
The matrix ${\si'}\inv\si$ takes the form
\[
 {\si'}\inv\si=\left(
\begin{matrix}
I_k&B\\
O&C
\end{matrix}
\right)
\]
where $O$ is $(n-k)\times k$ zero matrix and $C$ is a
$(n-k)\times (n-k)$-unimodular matrix.
Put $B=(b_{i,j})$.
The toric coordinates are related by
\begin{eqnarray}\label{coordinate-change}
\begin{cases}
& (u_{\si', k+1},\dots, u_{\si', n})=\pi_C(u_{\si, k+1},\dots, u_{\si, n}),\\
&u_{\si', i}=u_{\si, i}\times \prod_{j=k+1}^n u_{\si, j}^{b_{i,j}},
\,i=1,\dots, k
\end{cases}
\end{eqnarray}
In particular, we have the commutative diagram
\[\begin{matrix}
 \BC_{\si}^n&\mapright{p}&\hat E(\tau)\cap \BC_{\si}^n\\
\mapdown{\pi_{{\si'}\inv\si}}&&\mapdown{\pi_C}\\
\BC_{\si'}^n&\mapright{p'}&\hat E(\tau)\cap \BC_{\si'}^n
\end{matrix}
\]
where $p,p'$ are the projections into $\hat E(\tau)$ defined by 
$p(\bfu_\si)=\bfu_\si'$, $p'(\bfu_{\si'})=\bfu_{\si'}'$
where
$\bfu'_\si=(u_{\si,k+1},\dots, u_{\si,n})$ and 
$\bfu_{\si'}'=(u_{\si',k+1},\dots, u_{\si',n})$.
\end{Proposition}
\subsection{Tubular neighborhoods of the exceptional divisors}
First we fix $C^\infty$ function $\rho(t)$
such that 
$\rho\equiv 1$ for $t\le R$ and monotone decreasing for $R\le t\le 2R$ and
$\rho\equiv 0$ for $t\ge 2R$. The number $R$ is  large enough
and will be chosen later.
For $\si=(P_1,\dots, P_n)\in \cK_{n-1}$, we define
$\rho_\si(\bfu_\si)=\rho(\|\bfu_\si\|)$.
For each exceptional divisor $S=\hat E(P)^*\in \cT$, we consider the set of
$(n-1)$-simplicies
$\cK_P=\{\si\in \cK_{n-1}\,|\, P\in \si\}$.
For each $\si,\si'\in \cK_P$, after ordering the vertices of $\si,\si'$ as
$\si=(P,P_2,\dots, P_n)$ and $\si'=(P,P_2',\dots, P_n')$, we define
the {\em distance function $\dist_P$ from $\hat E(P)$} by
\[\begin{split}
\dist_P: X\to \BR,\,\,\,& \dist_{P }(\bfw)=\sum_{\si\in \cK_P} \dist_{P,\si}
(\bfw)\,\,\\
&\quad \dist_{P,\si}(\bfw):=\rho(\bfu_{\si}(\bfw))|u_{\si,1}(\bfw)|
\end{split}
\]
where 
$\bfu_\si(\bfw)$ is the coordinate of $\bfw$ in $\BC_\si^n$
and $\bfu_\si':=(u_{\si,2},\dots,u_{\si, n})$.
Put 
$ B_{K\,\si}(P)=\{(0,\bfu_{\si}')\,\,|\,\|\bfu_{\si}'\|\le K\}$.
We assume that $R$ is sufficiently large so that 
$\cup _{\si'\in \cK_P} B_{R,\si'}(P)=\hat E(P)$.
Note that the distance function 
 is  continuous on $X$ and $C^\infty$ on $X\setminus \hat E(P)$.
 We put
\[
 N_\eps(\hat E(P)):=\dist_P\inv([0,\eps]).
\]
\begin{Lemma}\label{dist} Suppose that  $\al=(\al_1, \al_2,\dots, \al_n)\in
 \BC_{\si}^n$ with $\sl\notin \pi\inv(\bf0)$ and $\al_1\ne 0$.
Put $\al(t)=(t\al_1,\al_2,\dots, \al_n)$ for $0\le t\le 1$,
$r_t=\dist_P(\al(t))$ and  $S(r_t):=\dist_P\inv(r_t)$.
Then $\al(0)\in \hat E(P)$ and 
$\lim_{t\to +0}T_{\al(t)}S(\al(t))$ is  the real  orthogonal space
 $v^\perp$
of the vector
$v:=(\al_1,0,\dots, 0)$. That is, the tangent space
$T_\al S(\al)$  converges to
the real hyperplane
  $v^\perp$ when  $t$ goes to zero.
\end{Lemma}
This lemma states that the tubular neighborhood $\partial N_\eps(\hat E(P))$
behaves infinitesimally as $|u_{\si,1}|=\text{constant}$. For the proof, see the Appendix (\S \ref{appendix}).

\subsubsection{Tubular neighborhood of $ \widetilde V(\tau)$}
Consider the stratum $ \widetilde V(\tau)$ with $\tau=(P_1,\dots, P_k)$.
Let $\cK_\tau$ be the set of coordinate charts $\si$ such that 
$\tau\prec \si$. We order the vertices so that 
$\si=(P_1,\dots, P_k,\dots, P_n)$.
We can write
\[\begin{cases}
& \pi_\si^*f(\bfu_\si)=u_{\si, 1}^{d(P_1)}\cdots u_{\si, k}^{d(P_k)} \widetilde f(\bfu_\si),\,\\
&\pi^*f_\De=u_{\si, 1}^{d(P_1)}\cdots u_{\si, k}^{d(P_k)} \widetilde f_\De(\bfu_\si')\,
\\
&\text{where}\,\,\, \widetilde f(\bfu_\si)= \widetilde f_{\De}(\bfu_\si')+
 R(\bfu_\si),\,\,\bfu_\si'=(u_{\si,k+1},\dots,u_{\si,n} )
\end{cases}
\]
The function $f_\De$ is by definition the face function   of the face
$\De:=\cap_{i=1}^k \De(P_i)$.
The second term  $R$ vanishes on $\hat E(\tau)$. Thus the polynomial
$ \widetilde f_{\De}$ is  a defining  polynomial of  $\tilde V(\tau)$
 in the coordinate chart $\BC_\si^n$.
Take another $\si'\in \cK_\tau$ and write ${\si'}\inv\si$ as in
(\ref{coordinate-change}).
Then we have
\[ \prod_{i=1}^k u_{\si,i}^{d(P_i)} \widetilde f_{\tau\si}(\bfu_\si')=
\prod_{i=1}^k u_{\si',i}^{d(P_i)} \widetilde f_{\tau\si'}(\bfu_{\si'}').\]
Thus we have
\begin{eqnarray}\label{eq-change}  \widetilde f_{\tau\si'}(\bfu_{\si'}')= \widetilde f_{\tau\si}(\bfu_\si')\times
 \prod_{i=k+1}^n u_{\si,j}^{m_j},\,\exists m_{j}\in \BZ.
\end{eqnarray}
Thus from now on, we 
fix an $(n-1)$-simplex  $\si=\si(\tau)$ for each $\tau$ and put 
\[
  \widetilde V_\eps(\tau)=\{\bfu_\si\in \hat E(\tau)\cap \BC_\si^n\,|\, | \widetilde
 f_{\tau\si}(\bfu_\si')|\le \sqrt\eps\}.
\]
We call $ \BC_{\si}^n$ {\em the canonical
coordinates chart of $\hat E(\tau)$.}
%
Now for each 
$\tau=(P_1,\dots, P_k)\in \cK$  such that  $\hat E(\tau)^*\subset \pi\inv({\bf 0})$,
we put
\[\begin{split}
 & N_\eps(\hat E(\tau))=\cap_{j=1}^k  N_\eps(\hat E(P_j)),\\
\end{split}
\]
where $ N_\eps(\hat E(\tau))$ is a tubular neighborhood of $ \widetilde E(\tau)$.
\subsubsection{Truncated tubular neighborhoods.}
Let 
$ p_{\tau\eps}:   N_\eps(\hat E(\tau))\to \hat E(\tau)$
 be the projection.
Recall that $p_{\tau\eps}$ is defined by the simple projection
$\bfu_\si\mapsto \bfu_\si'$ for any chart $\BC_\si^n$
with $\si=(P_1,\dots, P_n)$.
Now we  define {\em truncated Milnor stratification} as follows.
The truncated strata 
and truncated
tubular neighborhoods
 for the Milnor fibration are defined by
\[
\begin{split}
&  N_\eps( \widetilde E(\tau))^{tr}=p_{\tau\eps}\inv( \widetilde
 E_\eps(\tau)^{tr}),\\
& N_\eps( \widetilde V(\tau))^{tr}=p_{\tau\eps}\inv( \widetilde V_\eps(\tau)^{tr}),\\
 &\text{where}\,\,
\begin{cases}
& \widetilde E_\eps(\tau)^{tr}=(\hat E(\tau)\setminus  N_\eps( \widetilde V(\tau))
\setminus\cup_{\tau\prec
  \tau'}( N_\eps(\hat E(\tau')) )\\
& \widetilde V_\eps(\tau)^{tr}= \widetilde V_\eps(\tau)\setminus
 \cup_{\tau\prec\tau'} N_\eps(\hat E(\tau'))\\
\end{cases}\\
\end{split}
\]
Thus  we can write 
$ N_\eps(S)^{tr}=p_{\tau\eps}\inv(S^{tr})$,
using the  notations
\begin{eqnarray*}
S^{tr}&=&\begin{cases}
 \widetilde E_\eps(\tau)^{tr},\,\,&S= \widetilde E(\tau)^*,\\
 \widetilde V_\eps(\tau)^{tr},\,\,&S= \widetilde V(\tau)^*
\end{cases}.\end{eqnarray*}
%
Note that
 $ \widetilde E_\eps(\tau)^{tr}, \widetilde V_\eps(\tau)^{tr}$ are
 relatively
compact subsets of $\hat E(\tau)^*$ which is homotopy equivalent
 to $ \widetilde E(\tau)^*$ and $ \widetilde V(\tau)$ respectively.
Put \[ N_\eps(\pi\inv({\bf 0})):=
     \cup_{\hat E(\tau)\subset \pi\inv({\bf 0})}\left(
 N_\eps( \widetilde E_\eps(\tau))^{tr}\cup  N_\eps( \widetilde V(\tau))
\right)
    \]
Note that $ N_\eps(\pi\inv({\bf 0}))$
is a homopy
equivalent cofinal system of the neighborhood
of $\pi\inv({\bf 0})$.
We consider the Milnor fibration over  $D_\de^*$ with $\de\ll\eps$:
\[
 \pi^*f:\, N_{\eps,\de}\to D_\de^*,\,\,
\,\,
 N_{\eps,\de}:=(\pi^*f)\inv(D_\de^*)\cap  N_\eps(\pi\inv({\bf 0})).
\]
\subsection{Recipe of the proof}
\noindent
Step 1. First we will show that the restriction of $ \pi^*f:\, N_{\eps,\de}\to
D_\de^*$
over each tubular neiborhood $ N_\eps(S)^{tr}$
is a  fibration in the
way that each 
fiber is transverse to the boundary of $ N_\eps(S)^{tr}$.
Thus the restriction to the boundaries  $\partial N_\eps(S)^{tr}$ is also a fibration.

\vspace{.2cm}\noindent
Step 2. Then using the additive formula for the Euler characteristic and the
corresponding
product formula for the zeta function (see \cite {Okabook}, Chapter 1),
the calculation of the zeta function of the Milnor fibration is reduced 
to the calculation of the Milnor fibration restricted to
 each $\eps$-tubular neighborhood $ N_\eps(S)^{tr}$. This fibration is again a
 locally product of the Milnor fibration of the restriction to the
 normal slice of $S^{tr}$ and the  stratum $S^{tr}$.

\vspace{.2cm}\noindent
Step 3. Finally we determine the set of  stata
 which contribute the zeta function
(Lemma \ref{Contribution}). They correspond bijectively to $\cup_{I}\cS_I$.

\vspace{.3cm}
We say that  a simplex $\tau=(P_1,\dots, P_k)$ is
{\em of a divisor type}
if (up to an ordering of the vertices) $P_1\in \cV^+$ and the other vertices
$\{P_2,\dots, P_k\}$ is a subset of the non-positive vertices $ \{E_1,\dots,
E_n\}$.
A simplex 
$\tau$ of a divisor type is  called  to be {\em of a maximal dimensional face }
if $\De(P_1)\cap \Ga(f^I)$
is a maximal dimensional face of  $\Ga(f^I)$ ({\em i.e.} $\dim\,
\De(P_1)=|I|-1$)
where
$I= \{i\,|\, E_i\notin \tau\}$.

For a subset $I\subset \{1,2,\dots,n\}$, consider the set of vertices 
$\cS_I'$ of $\Si^*$ such that
there exists an $(n-|I|)$-simplex $\tau$ of a divisor type
with a maximal dimensioanl face ( $\tau\in \cK_{n-|I|}$)
whose vertices are $\{P,E_j\,|\, j\notin I\}$.
The key assertion is the following.
\begin{Lemma}\label{Contribution}
Take a stratum $S\in \cM$.

\noindent
(1) The Milnor fibration is
 decomposed into the fibrations restricted on $ N_{\eps}(S)^{tr}$
for each $S\in \cM$. This fibration is topologically 
determined by the corresponding face function.

\noindent
(2) The zeta function of the normal slice is
non-trivial only if  $S= \widetilde E(\tau)^*$ and  $\tau=(P_1,\dots, P_k)$
is of a divisor type.

\noindent
(3)
 The zeta function of the tubular neighborhood
$ N_\eps(S)^{tr}$ is non-trivial if and only if $\tau$
is of a maximal dimensional face.

\noindent
(4)
There is a bijective correspondance from  $\cS_I'$
 to $\cS_I$.
\end{Lemma}
Here the normal slice for $S= \widetilde V(\tau)$ implies normal plane of 
$ \widetilde V(\tau)$ in the fixed coordinate chart $\si(\tau)$
and the standard metric in this affine space.
%
The proof of Lemma \ref{Contribution} occupies the rest of this 
subsection.

Recall that in the Milnor stratification $\cM$, there are two type of strata:
$ \widetilde  E(\tau)^*$ and $ \widetilde V(\tau)^*$
 wuth $\tau=(P_1,\dots, P_k)$.
Let $\ell(\tau)=\sharp\{i\,|\, d(P_i)>0\}$ and we refer $\ell(\tau)$
{\em the strict positivity dimension} of $\tau$.
Take a $(n-1)$-simplex $\si=(P_1,\dots, P_n)$ having $\tau$ as a face.
We write   $(u_1,\dots, u_n)$ for simplicity instead of 
$(u_{\si, 1},\dots, u_{\si, n})$, 
 the canonical  toric coordinates of $\BC_\si^n$  
and 
$\si=(P_1,\dots, P_n).$

\noindent
{\bf Case 1. $S= \widetilde E(\tau)$}
\newline
As $\hat E(\tau)\subset \pi\inv({\bf 0})$, we may assume that 
$P_1\in \cV^+$ so that $d(P_1)>0$ hereafter.
 Put $\De:=\cap_{i=1}^k\De(P_i)$. 
$\pi_\si^*f$ take the form
\begin{eqnarray*}
&\begin{cases}
& \pi^*f= U_{1,k} \widetilde f(\bfu),\,  \text{where}\,\,U_{1,k}:=u_{1}^{d(P_1)}\cdots u_{ k}^{d(P_k)},\\
&\pi^*f_\De= U_{1,k} \widetilde f_\De(\bfu),\,
\,\, \widetilde f(\bfu)= \widetilde f_\De(\bfu')+ R(\bfu)
\end{cases}
\end{eqnarray*}
where $R$
is contained in the ideal $(u_1,\dots, u_k)$
and therefore it vanishes on $\hat E(\tau)\cap \BC_\si^n$
and  $\bfu'=(u_{k+1},\dots, u_{ n})$ are the coordinates of $\hat
E(\tau)\cap \BC_\si^n$. 
We consider the homotopy
$\hat f_t=  U_{1,k} \widetilde f_t$ for $0\le t\le 1$ where 
 $ \widetilde f_t(\bfu):= \widetilde f_\De(\bfu')+ t
 R(\bfu)$.
Note that $\hat f_1=\pi^*f$ and $\hat f_0=\pi^*f_\De$
is associated to the face function 
$f_\De$.

\subsubsection{Smoothness.}
Consider the family of the restriction of
Milnor fibering $\hat f_t: N_\eps(S^{tr})\cap {\hat f_t}\inv(D_\de^*)\to D_\de^*,\,\de\ll \eps$
and their
 Milnor fibers $F_{t,\de}(\tau):=\hat f_t\inv(\de)\cap
 N_\eps(S^{tr})$ and $S= \widetilde E(\tau)^*$.
By the non-degeneracy assumption, (a submatrix
of) the jacobian matrix
\[\begin{split}
 (\sharp):\,J:=&\left(
\frac{\partial \hat f_t}{\partial  u_{k+1}}(P),\dots, \frac{\partial
   \hat f_t}{\partial  u_{n}}(\bfu)\right)\\
&= U_{1,k} \left(\frac{\partial  \widetilde
   f_\De}{\partial  u_{k+1}}(\bfu')+t\frac{\partial R}{\partial  u_{k+1}}(\bfu),\dots, \frac{\partial
    \widetilde f_\De}{\partial  u_{n}}(\bfu)+t\frac{\partial R}{\partial  u_{n}}(\bfu)\right).
\end{split}
\]
By the non-degeneracy, there exists $k+1\le j\le n$ such that 
$\frac{\partial  \widetilde
   f_\De}{\partial  u_{j}}(\bfu')\ne 0$. 
As $R$ and $\frac{\partial R}{\partial    u_j},\, j\ge k+1$
are constantly zero on $\hat E(\tau)$, this implies that 
$\frac{\partial \hat
   f_t}{\partial  u_{j}}(\bfu')\ne 0$ for sufficiently small $\eps$ and
   $\de\ll \eps,\,i\le k$.
Thus $J\ne (0,\dots, 0)$
for any $u\in F_{t,\de}(\tau)$  with $\bfu'\in S^{tr}$,
as $S^{tr}$ is relatively compact. 
 So $F_{t,\de}(\tau)$ is also smooth.

\subsubsection{ Transversality.} We consider the transversality
of $F_{t,\de}$ and the boundary  $\partial( N_{\eps}(S^{tr}))$ at the intersection
of $B:=\partial( N_{\eps}(S^{tr}))\cap  N_{\eps}S'$
or $B'=\partial( N_{\eps}(S^{tr}))\cap  \widetilde V_{\eps}(S')$
where $S'= \widetilde E(\tau')^*$ with 
$\tau'=(P_1,\dots, P_k,\dots, P_m)$.
Put $\De'=\cap_{i=1}^m\De(P_i)$.
Let $\si(\tau')$ be the fixed chart for $\tau'$ and let
$\bfv=(v_1,\dots,v_n)$
be the toric coordinates of $\BC_{\si(\tau')}^n$ for simplicity.
 As we are considering a tubular neighborhood
 of polydisk type,
\[\begin{split}
B=&\{\bfv\,|\, |\dist_{P_i}(\bfv)|=\eps,\, k+1\le i\le m\}\quad
\text{or}\\
B'=&\{\bfv\,|\, |\dist_{P_i}(\bfv)|=\eps,\, k+1\le i\le m,\, | \widetilde
   f_{\De'}(\bfv'')|=
\sqrt\eps\}.
\end{split}
\]
By (\ref{eq-change}) for some integers $m_i,i=m+1,\dots, n$,
\[
 f_{\De'}(\bfv'')=f_{\De'}(\bfu'')\prod_{i=m+1}^n u_i^{m_i},
\,\,m_i\in \BZ,m+1\le i\le n.
\]
Putting $ U_{k+1,m}=\prod_{j=k+1}^m u_j^{d(P_j)}$, we can write further that
\begin{eqnarray}
\begin{cases}\label{transverse}
&\hat f_t(\bfu)= U_{1,k} \widetilde f_t(\bfu)\\
& \widetilde f_t(\bfu)= \widetilde f_\De(\bfu')+t R(\bfu)\\
&\qquad = U_{k+1,m}\left( \widetilde f_{\De'}(\bfu'')+R'(\bfu')+ t\bar R(\bfu)
\right) \\
&  \widetilde  f_\De(\bfu')= U_{k+1,m}\left( \widetilde
		     f_{\De'}(\bfu'')+R'(\bfu')\right),\\
&\hat f_t(\bfu)=U_{1,k}U_{k+1,m}
\left( \widetilde f_{\De'}(\bfu'')+R'(\bfu')+ t\bar R(\bfu)
\right),\,\\
&R= U_{1,k} U_{k+1,m} \bar R
\end{cases}
\end{eqnarray}
where $\bfu''=(u_{m+1},\cdots, u_{ n})$ and $R'$  in the ideal generated
by
$u_{k+1},\dots, u_{m}$. Note that 
$\bar R(\bfu)$ is in the ideal generated by $\{u_{i}u_j\,|\, i\le k<j\le m\}$.
Thus the transversality follows from the fact that 
 the jacobian  submatrix
\[
\frac{ \partial(\Re\hat f_t,\Im\Re\hat f_t,
\dist_{P_{k+1}},\dots,\dist_{P_m} )}{\partial(x_{ \si 1},y_{\si 1}\dots,x_{\si m}
,y_{\si m})},\quad
u_j=x_{\si 1}+  i y_{\si j}
\]
has rank $m-k+2$.
For the proof of this assertion, we use the polar
coordinates as follows.

Assume that $S'=V(\tau')$ is non-empty, {\it i.e.}, namely $\dim\,\De'\ge 1$.
Put $g:= \widetilde f_{\De'}(\bfv'')$.
On a neighborhood of a chosen point $\bfu_0''\in
\partial \widetilde V_\eps(\tau')^{tr}$, by 
the non-degeneracy of $f$ on $\De'$,
$g$ can be used as a member of a coordinate chart.
 For example, we may assume that
there exists an open neighborhood $U(\bfu^0),\,\bfu^0=(u_{1}^0,\dots,u_{ m}^0,\bfu_0'')$ such that 
$(u_1,\dots, u_m, g, v_{m+2},\dots, v_{n})$  is a coordinate chart on
$U(\bfu^0)$ and $(g, v_{m+2},\dots, v_{ n})$ is a coordinate chart of
$U(\bfu^0)\cap \hat E(\tau')$.
We use  the polar coordinates for $u_{k+1},\dots, u_m$, $g$.
So put 
\[\begin{split}
&u_1=x_{\si1}+  i  y_{\si 1}\\
& u_j=\rho_je^{\theta_j   i },\,j=k+1,\dots, m,\quad g=\rho_g
e^{\theta_g  i }\\
\end{split}
\]
The tranversality 
can be checked
using the subjacobian with respect to $(x_{\si1},y_{\si1},\rho_{k+1},\dots, \rho_m,\rho_g )$:
\begin{Assertion}\label{Assertion}
Under the above notations, we have
\[
\rank \left(
\frac{\partial(\Re\hat f_t,\Im \hat f_t,\dist_{P_k+1},\dots,
 \dist_{P_m},\rho_g)}{\partial(x_{\si 1},y_{\si 1},\rho_{k+1},\dots, \rho_m,\rho_g)}
\right)=m-k+3.
\]
\end{Assertion}
For the proof, see \S \ref{transversality-proof} in Appendix \ref{appendix}.
\begin{Remark}
Assume $d(P_i)>0$ for some $2\le i\le k$.
Then the Milnor fiber $F_{t,\de}$ also intersects with the boundary of the tubular neighborhood:
 $\dist_{P_i}=\eps$.
The transversality with this boundary is treated considering it as  the
 transversality of the Milnor fiber in the stratum
$ \widetilde E(\tau')^{tr}$, 
$\tau'=\tau\setminus\{P_i\}$ with $\partial  N_\eps( \widetilde E(\tau))^{tr}$.
Thus this case is treated in the pair $\tau'\prec \tau$.
\end{Remark}

Thus we have observed that the Milnor fibration of $\pi^*f$
is the union of fibrations restricted on 
$ N_\eps (S)^{tr}$ and $\partial  N_\eps (S)^{tr}$
with $S= \widetilde E(\tau)^{tr}$ or 
$S= \widetilde V(\tau)^{tr}$ with $\tau=(P_1,\dots, P_m)$. 
Using the homotopy $\pi_\si^*f_t$, this restriction is equivalent to the
fibration
defined by $\pi_\si^*f_\De$ where $\De=\cap _{i=1}^k \De(P_i)$.
This proves the first assertion (1).

\subsection{Zeta functions}
Next, we consider the assertion for the zeta function (2).
\subsubsection{Case 1. Stratum $ \widetilde V_\eps(\tau)$.}
We first consider the stratum $ \widetilde V_\eps(\tau)$:
Let $\psi_S: N_{\eps}( \widetilde V_\eps(\tau))\to \widetilde V_\eps(\tau) $
be the projection of the tubular neighborhood.
At each point $x\in \widetilde V(\tau)^{tr} $, the
Milnor fibration is homotopically defined  by $\pi^*f_{\De}=\hat f_\De$ 
($\De=\cap_{i=1}^k\De(P_i)$).
Recall that 
\[
\hat f_\De (\bfu)=\prod_{i=1}^k u_i^{d(P_i)}\times  \widetilde f_\De(\bfu')
\]
Put $g= \widetilde f_\De(\bfu')$. 
Take a point $x\in  \widetilde V(\tau)$.
 Assume $\frac{\partial  \widetilde f_\De(\bfu')}{\partial
 u_{k+1}}(x)\ne 0$ for example.
Then we may assume that locally $(g,u_{k+2},\dots, u_n)$ is a coordinate
system of
a neighborhood, say $U(x)$ of 
$x\in  \widetilde V(\tau)$ and
also $(u_1,\dots, u_k,g,u_{k+2},\dots, u_n)$
is a coordinate system of the open set
$\cap_{i=1}^k\{|u_i|< \eta\}\cap \psi_S\inv( U(x))$.
By the relative compactness of the trancated strata $ \widetilde
V(\tau)^{tr}$,
we may also assume  that $\eps\ll\eta$
so that 
\[
 \cap_{i=1}^k\{\dist_{P_i}\le \eps\}\cap
\psi_S\inv(U(x))\subset\cap_{i=1}^k\{|u_i|< \eta\}\cap \psi_S\inv(
U(x)).
\]
 In the normal slice of $x$, $u_1,\dots, u_k, g$ are coordinates.
The Milnor fiber restricted on 
$ N_\eps( \widetilde V(\tau)^{tr})$ is 
locally  equivalent to the  product of $U(x)$
 and the Milnor
fibration of the polynomial
$h=g\prod_{i=1}^k u_{i}^{d(P_i)} $ (=the definig polynomial in the normal slice) in $\BC^{k+1}$.
Namely the fibration
\[
\hat f_\De:  N_\eps( \widetilde V(\tau)^{tr} )\cap 
 N_{\eps,\de}\cap \psi_S\inv (U(x))\to D_\de^*
\]
is isomorphic to the product of $U(x)$ and  the restriction to the normal slice:
 \[
   \widetilde f_{\De,x}: \psi_S\inv(x)\cap \hat
  N_{\eps,\de}\to D_\de^*.
\]
We  consider two tubular neighborhoods:
\[
 \begin{split}
& N_\eps(\hat E(\tau)^{tr})=\{\bfu\,|\, \dist_{P_i}(\bfu) \le
  \eps,\,i=1,\dots, k\}\\
& N_\eps(\hat E(\tau)^{tr})'=\{\bfu\,|\,|u_i|\le \eps,\,i=1,\dots, k\}.
\end{split}
\]
By the cofinal homotopy equivalent argument, we can consider the normal
Milnor fibration in the latter space and we see easily that
the fiber is given by 
\[
 F_{\de}=\{(u_1,\dots, u_k,g)\,|\,g\prod_{i=1}^k u_i^{d(P_i)}=\de,\, |g|\le \sqrt{\eps},\,|u_i|\le \eps\}
\]
 and it is homotopic to
$ (S^1)^\ell$
where $\de\ll\eps$ and  $\ell$ is the strict positivity dimension of
$\tau$.
As $ \widetilde V(\tau)\subset \pi\inv({\bf 0})$, $\ell\ge 1$.
Thus the Euler characteristic of $F_\de$ is also zero and  the monodromy is trivial.
There are no contribution from this statum  to the zeta function.

\subsubsection{ Case 2. Stratum $S= \widetilde E(\tau)^{tr}$.}
Now we consider the Milnor fibering  on the tubular neighborhood 
over the strata 
$S= \widetilde E(\tau)^{tr}$.
We have seen that the Milnor fibration of $\hat f$ is again
isomorphic to the fibration defined by  $\hat f_\De$
and  the latter is  locally product of the base space and the Milnor
fibration of the restriction to  the normal slice. This normal slice function
is locally described by the product function
$u_{1}^{d(P_1)}\cdots u_{ k}^{d(P_k)}f_\De(\bfu')$.
The factor $f_\De(\bfu')$ is a constant on the normal  slice
$\bfu'={\rm const}$.
We know that the fiber in this normal slice
is homotopic to $\gcd(d(P_1),\dots, d(P_k))$ copies of 
$(S^1)^{\ell-1}$ where $\ell$ is the strict positivity dimension of
$\tau$. 
See for example \cite{Okabook}.
 Therefore {\em the Euler characteristic of this slice Milnor fiber is 
non-zero
if and only if $\ell=1$}.
This implies $\tau$ is of a divisor type.
Assume  for example  that 
$d(P_1)>0$ and $d(P_i)=0,\,2\le i\le k$. This implies $P_i=E_{\nu(i)}$
for $i=2,\dots, k$ and $\tau=(P_1,E_{\nu(2)},\dots,E_{\nu(k)})$
and $\De(P_1;f^I)$ is a face of $\Ga(f^I)$ where 
$I=\{1,\dots,n \}\setminus\{\nu(2),\dots, \nu(k)\}$.
The Milnor fiber $F_{\tau}^*$ restricted to this stratum
is defined by 
\[
 F_{\tau}^*=\{
\bfu\,|\, u_1^{d(P_1)} \widetilde f_\De(\bfu')=\de,\,
\bfu'\in  \widetilde E(\tau)^{tr}\}
\] is hototopically
$d(P_1)$ disjoint polydisk of dimension $k-1$
 defined by $\{u_{ 1}=\de^{1/d(P_1)},|u_j|\le\eps ,j=2,\dots, k\}$
over $(0,\bfu')$
with $\bfu'=(u_{k+1},\dots,u_{\si n})$. 
Thus the  zeta function of the Milnor fibration of the normal slice is 
$(1-t^{d(P_1)})$ and thus by a standard Mayer-Vietoris argument,
 the zeta function of the Milnor fibration over this stratum
is 
\begin{eqnarray}\label{zeta-formula}
\zeta_S(t)=(1-t^{d(P_1)})^{\chi( \widetilde E(\tau)^*)}.
\end{eqnarray}
Now the proof of Lemma \ref{Contribution} reduces to:
\begin{Lemma} Assume that  $\tau=(P_1,E_{\nu(1)},\dots,  E_{\nu(k)})$ as above.
Then  $\chi( \widetilde E(\tau)^*)\ne 0$   
only if $\De:=\De(P_1)\cap \BR^I$ is a 
face of maximal dimension of  $\Ga(f^I)$
where  $I=\{1,\dots, n\}\setminus\{\nu(1),\dots\nu(k)\}$.
\end{Lemma}
\begin{proof}
Put $I=\{1,\dots, n\}\setminus\{\nu(1),\dots,\nu(k)\}$.
Then
\[
\chi( \widetilde E(\tau)^*)=\chi(\hat E(\tau)^*\setminus  \widetilde V(\tau))=
-\chi( \widetilde V(\tau)^*)
\]
 as $\hat E(\tau)^*\cong \BC^{*(n-k)}$.
As $\si=(P_1,\dots, P_n)$
is a unimodular matrix, $P:=P_1^I$ is a primitive vector
and $\De=\De(P_1)\cap \BR^I$ and $f_\De$ is nothing but 
$(f^I)_{P}$.
\[\begin{split}
  \widetilde V(\tau)^*&=\{\bfu=(u_{1},\dots, u_{n})\,|\, \bfu\in \hat E(\tau)^*,\,
 \widetilde f_{P_1}(\bfu')=0\}\\
&\overset{\pi_\si} \cong\{\bfz^I\in \BC^{*I}\,|\, f_{P}^I(\bfz^I)=0\}\\
&=\{\bfz^I\in \BC^{*I}\,|\, f_{\De}^I(\bfz^I)=0\}\\
\end{split}
\]
where $\si$ is a $(n-1)$-simplex with $\si=(P_1,\dots, P_n)$
and the Euler characteristic of the 
variety
$\{\bfz^I\in \BC^{*I}\,|\, f_{\De}^I(\bfz^I)=0\}$
is non-zero if only if $\dim\,
 \De=n-k$ with $\De= \De(P_1)\cap \BR^I$ (See for example, Theorem
   (5.3), \cite{Oka-newton2}).
\end{proof}
%
 \subsubsection{Correspondence of $\cS_I$ and $\cS_I'$.}
For a fixed $I\subset \{1,\dots,n\}$ with $|I|=k$, let us consider the set of vertices
$ \cS_I'$ which is the set of vertices $P\in \cV^+$ such that 
there exists a simplex $\tau=(P_0,\dots, P_k)\in \cK_{k},\,P_0=P$
of a divisor type with a maximal  face such that 
$P_j=E_{\nu(j)}$ for $j=1,\dots, k$, $I=\{\nu(i),\,i=1,\dots, k\}$ and 
$\De= \De(P_1)\cap\BR^I$.
As $\tau$ is a regular simplex, the $I$ component $P^I$ of $P$ is a primitive
vector
such that $\De(P^I,f^I)=\De$ and $d(P^I,f^I)=  d(P,f)$.
The proof of Theorem \ref{Varchenko}  is now completed by the following.
\begin{Proposition} There is a one-to one correspondence 
of $\cS_I'$ and $\cS_I$ by
\[
\xi:\cS_I'\to \cS_I,\,
 P\mapsto P^I.
\]
\end{Proposition}
\begin{proof}We check the surjectivity of $\xi$.
Take a face $\Xi$ of $\Ga(f^I)$
of maximal dimension.
Consider the set of covectors
$ \Si^*(\Xi)=\{P\,|\, \De(P)\supset \Xi\}$. 
It is obvious that 
$E_i\in  \Si^*(\Xi)$  if $i\notin I$.
 Then there exists a  vertex $P\gg 0$ of $\Si^*$
such that $\{P, E_i\,|\, i\notin I\}$ is a simplex of $\cK$.
Then 
 $\De(P)\cap \BR^I=\Xi$ and $P^I$ is primitive. 
Thus $P^I\in \cS_I$.
Assume that $P,Q\in \Si^*(\Xi)$. The cone of $\Si^*(\Xi)$  has $n-|I|+1$
 dimension.
The cones spanned by $\{P, E_i\,|\, i\notin I\}$
and $\{Q, E_i\,|\, i\notin I\}$ have dimension $n-|I|+1$ and thus they
 must contain an open subset in their intersection.
This is only possible if $P=Q$. This proves the injectivity.
\end{proof}
\subsection{Appendix}\label{appendix}
\subsubsection{Proof of Lemma \ref{dist}.}
Put $J=\{j\,|\,\al_j=0\}$ and $I=\{1,\dots, n\}\setminus J$.
As $\al\notin \pi\inv(\bf0)$, this implies $\{E_j\,|\, j\in J\}$ are
vertices
of $\si$. Then we restrict the argument to $\BC^I$ and $f^I$. Thus 
we may assume for simplicity that  $\al_j\ne 0$ for
 $1\le j\le n$. Then using the equality  
$u_{\si',1}=u_{\si,1}\prod_{j=2}^n u_{\si,j}^{b_{1,j}}$ and  putting 
$\rho_1=|u_{\si,1}|=\sqrt{x_{\si,1}^2+y_{\si,1}^2}$, we get 
\begin{eqnarray*}
&\frac{\partial \dist_{P,\si'}(\bfu_{\si'})}{\partial x_{\si,1}}\big |_{{\bfu_\si'=\al'}}=\frac{x_{\si,1}}{\rho_1}\rho(\al_{\si'})\prod_{j=2}^n
|\al_{j}|^{b_{1 j}}+o(\rho_1),\\
&\frac{\partial \dist_{P,\si'}(\bfu_{\si'})}{\partial y_{\si,1}}\big |_{{\bfu_\si'=\al'}}=\frac{y_{\si,1}}{\rho_1}\rho(\al_{\si'})\prod_{j=2}^n
|\al_{j}|^{b_{1 j}}+o(\rho_1),\,\\
&\frac{\partial \dist_{P,\si'}(\bfu_{\si'})}{\partial x_{\si,j}}\big |_{{\bfu_\si'=\al'}}=o(\rho_1),\,
\frac{\partial \dist_{P,\si'}(\bfu_{\si'})}{\partial y_{\si,j}}\big
|_{{\bfu_\si'=\al'}}=o(\rho_1),\,\,2\le j\le n
\end{eqnarray*}
where
$\al_{\si'}=(\al_1',\al_2',\dots, \al_{ n}')=\pi_{{\si'}\inv\si}(\al)$
and  $u_{\si,j}=x_{\si,j}+y_{\si,j}\,  i $.
Here 
$o(r_1)$ is by definition  a smaller term than $\rho_1$ when $\rho_1\to 0$.
This implies that  in  $\BC_\si^n$
with real coordinates  $(x_{\si,1},y_{\si,1},\dots, x_{\si,n},y_{\si\, n})$,
\[\begin{split}
&\grad\, \dist_P(\al(t))=(\be\frac{a_1}{|\al_1|},\be\frac{b_1}{|\al_1|},0,\dots,0)+O(t),\,\text{namely}\\
& \grad\, \dist_P(\al(t))\overset{t\to +0}\longrightarrow\,
   (\be\frac{a_1}{|\al_1|},\be\frac{b_1}{|\al_1|},0,\dots,0)
\end{split}
\] 
with $\al_1=a_1+b_1   i $,  
\[\be=\rho(\al_\si)+\sum_{\si'\in \cK_P,\si'\ne \si}\rho(\al_{\si'})\prod_{j=k+1}^n
|\al_{j}'|^{b_{1 j}}>0.\]
This proves the assertion.
\subsubsection{Proof of Assertion \ref{Assertion}.}\label{transversality-proof}
Fix a point $\bfu_0\in \partial \tilde V_\eps(\tau')^{tr}$. 
Put $U_{1,m}=\prod_{i=1}^m u_i^{d(P_i)}$ and $u_1=x_{\si 1}+  i  y_{\si 1}$.
Recall that 
\[
\hat f_t(\bfu)= \tilde f_t(\bfu)\prod_{j=1}^{m} u_j^{d(P_j)},\quad
\tilde f_t(\bfu)=\tilde f_{\De'}(\bfu'')+R'(\bfu')+t\bar R(\bfu).
\]
Recall that $R'(\bfu')$ does not contain $u_1,\dots, u_k$ and  contained in the ideal generated by 
$(u_{k+1},\dots, u_m)$ and 
$\bar R$ is contained  in the ideal generated by
$u_iu_j,\, i\le k<j\le m$. First note that 
\begin{eqnarray*}\label{tubular}
& \frac{\partial \hat f_t}{\partial x_{\si 1}}=U_{1,m}J_{x_{\si1}}(\bfu),\,\,
 \frac{\partial \hat f_t}{\partial y_{\si 1}}=U_{1,m}
J_{y_{\si 1}}(\bfu)\\
& \frac{\partial \hat f_t}{\partial r_i}=U_{1,m}
J_{i}(\bfu)\qquad  i=k+1,\dots, m, \, g\\
&J_{x_{\si 1}}(\bfu)=
\frac {d(P_1) }{u_1}\left( \widetilde f_{\De'}(\bfu'')+R'(\bfu')+
t\bar R(\bfu)\right)+
t\frac{\partial \bar R(\bfu)}{\partial u_1},\\
&J_{y_{\si 2}}(\bfu)=  i \frac {d(P_1) }{u_1}\left( \widetilde f_{\De'}(\bfu'')+R'(\bfu')+
t\bar R(\bfu)\right)+
  i t\frac{\partial \bar R(\bfu)}{\partial u_1},\\
&J_i(\bfu)=\frac {d(P_i) }{r_i}\left( \widetilde f_{\De'}(\bfu'')+R'(\bfu')+
t\bar R(\bfu)\right)+\left(\frac{\partial \bar R'(\bfu')}{\partial r_i}+
t\frac{\partial \bar R(\bfu)}{\partial r_i}
\right),\, k+1\le i \le m\\
&J_g(\bfu)= 1+t\frac{\partial\bar R(\bfu)}{\partial r_g}.
\end{eqnarray*}
Note that 
the respective orders of $ \widetilde f_{\De'}(\bfu'')$  and $r_j,\,k+1\le j\le m$
are $\sqrt{\eps}$ and $\eps$.
The second term of $J_i,\,i\ne g$ is at most $\eps$.
Thus the  main term of $J_i$ is  $\frac{d(P_i)\tilde
f_{\De'}(\bfu'')}{r_i}$
 and
the   order is
$\frac{1}{\sqrt{\eps}}$.
The order of $J_g$ is 1.
Thus we observe that
\begin{eqnarray*}
 X_{\BC}&:=( \frac{\partial \hat f_t}{\partial x_{\si 1}},\frac{\partial \hat f_t}{\partial y_{\si 1}}, \frac{\partial
\hat f_t}{\partial r_{k+1}},\dots,\frac{\partial \hat f_t}{\partial
r_m},\frac{\partial \hat f_t}{\partial r_g})\\
&\prosim
(\frac{d(P_1)}{u_1},  i \frac{d(P_1)}{u_1},\frac{d(P_m)}{r_{k+1}}\dots,\frac{d(P_m)}{r_m},0).
\end{eqnarray*}
Here $\bfa\sim \bfb$ implies $\lim_{\eps\to 0}\bfa/\|\bfa\|=\lim_{\eps\to
0}\bfb/\|\bfb\|$
in the projective space.
$X_{\BC}$ is a complex vector. As a real matrix, this corresponds to $2\times (m-k+3)$
real matrix:
\[
X_{\BR}=\left(\begin{matrix}
&\frac{\partial\Re \hat f_t}{\partial x_{\si1}},\frac{\partial \Re\hat f_t}{\partial y_{\si1}}, \frac{\partial
\Re\hat f_t}{\partial r_{k+1}},\dots,\frac{\Re\partial \hat f_t}{\partial
r_m},\frac{\partial \Re\hat f_t}{\partial r_g}\\
&\frac{\partial \Im\hat f_t}{\partial x_{\si1}},\frac{\partial \Im\hat f_t}{\partial y_{\si1}}, \frac{\partial
\Im\hat f_t}{\partial r_{k+1}},\dots,\frac{\partial\Im \hat f_t}{\partial
r_m},\frac{\partial\Im \hat f_t}{\partial r_g}
\end{matrix}\right)=\left(\begin{matrix}X_1\\X_2
\end{matrix}
\right)
\]
by $X_{\BC}=X_1+  i X_2$.
Let us write the first $2\times 2$-minor  of this matrix:
\[
\left(
\begin{matrix}
x_{11}&x_{12}\\
x_{21}&x_{22}
\end{matrix}
\right)=
\left(
\begin{matrix}
\frac{\partial\Re\hat f_t}{\partial x_{\si1}}&\frac{\partial\Re\hat f_t}{\partial y_{\si1}}\\
\frac{\partial\Im\hat f_t}{\partial x_{\si1}}&\frac{\partial\Im\hat f_t}{\partial y_{\si1}}\\
\end{matrix}
\right)
\]
That is,
\[
(x_{11}+  i x_{21},x_{12}+  i  x_{22})=
(\frac{\partial\hat f_t}{\partial x_{\si1}},\frac{\partial\hat f_t}{\partial y_{\si1}}).
\]
By Lemma \ref{dist}, we see also that 
\[
 \left(\frac{\partial\dist_{P_i}}{  \partial x_{\si1}},\frac{\partial\dist_{P_i}}{  \partial y_{\si1}},\frac{\partial\dist_{P_i}}{\partial r_{k+1}},\dots, 
\frac{\partial\dist_{P_i}}{\partial r_m},
\frac{\partial\dist_{P_i}}{\partial r_g}\right)\sim (0,\dots,1,\dots,0)
\]
for any $k+1\le i\le m$.

 Therefore the rank of
 Jacobian matrix is infinitesimally equivalent to the rank of
\[\begin{split}
&\frac{\partial(\Re \hat f_t,\Im\hat f_t,\dist_{P_{k+1}},\dots,\dist_{P_m},r_g)}
{\partial(x_{\si1},y_{\si1},r_{k+1},\dots, r_m,r_g)}\\
&\prosim
 \left(
\begin{matrix}
\Re(d(P_1)/u_1)& \Re(  i \, d(P_1)/u_1)&*\\
\Im(d(P_1)/u_1)& \Im(  i \, d(P_1)/u_1)&*\\
0&0&I_{m-k+1}\\
\end{matrix}
\right)
\end{split}
\]
in the projective sense for each raw vectors.
Put $d(P_1)/u_1=\al+  i \be$.
Then the upper left $2\times 2$-matrix  $A$ is
nothing but
\[
A=\left(
\begin{matrix}
\al&-\be\\
\be&\al
\end{matrix}
\right)
\]
and $\det\,A=\al^2+\be^2\ne 0$.
This is enough to see that the rank of the Jacobian is $m-k+3$.
This proves the  transversality.
\subsubsection{Cofinal homotopy equivalent   sequence.}
Let $X$ be a manifold  and  let  $\{E_j\},\{ E_j'\}, j\in \BN$ be a
decreasing 
sequence of submanifolds such that 
the inclusions $\iota_{j+1}:E_{j+1}\hookrightarrow E_j$,
and $\iota_{j+1}':E_{j+1}'\hookrightarrow E_j'$
are homotopy equivalence.
Suppose further that
\[
 E_{j+1}'\subset E_{j+1}\subset E_j'\subset E_j.
\]
Then  the inclusions $\xi_{j+1}:E_{j+1}\hookrightarrow E_j'$ and 
$ \xi_{j+1}':E_{j+1}'\subset E_{j+1}$
are  homotopy equivalences 
as $\xi_{j+1}\circ \xi_{j+1}'=\iota_{j+1}'$ and $\xi_{j}'\circ\xi_{j+1}=\iota_{j+1}$.
Furthermore suppose that they are total spaces of 
fibrations over the same base space $Z$ with the commutativity in the following diagram:
\[
 \begin{matrix}
E_{j+1}'&\hookrightarrow&E_{j+1}&\hookrightarrow
&E_j'&\hookrightarrow&E_j\\
\mapdown{p}&&\mapdown{q}&&\mapdown{p}&&\mapdown{q}\\
Z&=&Z&=&Z&=&Z
\end{matrix}
\]
Then the corresponding fibrations are also homotopy equivalent.
We refer this argument as
{\em a cofinal homotopy equivalent   sequence argument}.
%
%
%
\section{Mixed functions}
\subsection{Main result}
Now we are ready to state our main result. First we prepare:
\begin{Proposition} Assume that $f(\bfz,\bar\bfz)$ is a
convenient mixed function of strongly polar positive weighted homogeneous face type.
Then for any weight vector $P$, $f_P$ is also a 
strongly  polar positive weighted homogenous polynomial
with weight $P$.
\end{Proposition}
\begin{proof}
The assertion is proved by the descending induction on $\dim\,\De(P)$.
The assertion for the  case $\dim\,\De(P)=n-1$ is the definition itself.
Suppose that $\dim\,\De(P)=k$ and the assertion is true for faces with $\dim\,\De\ge k+1$.
In the dual Newton diagram, $P$ is contained in the interior of a cell, say
$\Xi$,  whose vertices $Q_1,\dots, Q_s$ satisfy
$\dim\,\De(Q_j)\ge k+1$ for $j=1,\dots, s$. This implies 
$P$ is a linear combination
$\sum_{j}^s a_j\,Q_j$ with $a_j\ge 0$ and $\dim\De(Q_j)\ge k+1$.
This also  implies that $\De(P)=\cap_{j=1}^s\, \De(Q_j)$.
Write $f_P(\bfz,\bar\bfz)=\sum_{k} c_k \bfz^{\nu_k}{\bar\bfz}^{\mu_k}$.
As $f_{Q_j}$ is polar weighted homogeneous polynomial with weight $Q_j$,
we have the equality:
$\pdeg_{Q_j} \bfz^{\nu_k}{\bar\bfz}^{\mu_k}=m_j$ for $ j=1,\dots,s$
where $m_j$ a positive integer which is  independent of $k$.
This implies 
$f_P$ is polar weighted homogenous polynomial of weight $P$ with polar degree
$\sum_{j=1}^s a_jm_j>0$.
\end{proof}
We take a regular convenient simplicial subdivision $\Si^*$ of  $\Ga^*(f)$ (=regular fan) and we consider
the toric modification $\pi:X\to \BC^n$ with respect to $\Si^*$ as in \S 2.
Let $\cS_I$ be as in \S 2.
Now we can generalize the theorem of Varchenko  for the mixed polynomial 
$f(\bfz,\bar \bfz)$ as follows.
\begin{Theorem} \label{main-result}
Let $f(\bfz,\bar\bfz)$ a convenient non-degenerate
mixed polynomial of strongly polar positive weighted homogeneous face type.
  Let $V=f\inv(V)$ be a germ of hypersurface 
at the origin and let $ \widetilde V$ be the strict stransform of $V$ to $X$.
Then

\noindent
(1) $\tilde V$ is topologically smooth and  real  analytic 
 smooth variety outside of $\pi\inv({\bf 0})$.

\noindent
(2) $ \widetilde V(\tau)^*$ is  a real  analytic smooth  mixed
     variety for any $\tau\in \cK$.

\noindent
(3)
The zeta function of the Milnor fibration  of $f(\bfz,\bar\bfz)$
is given by the formula
\[\begin{split}
 &\zeta(t)=\prod_I
\zeta_I(t),\,
 \zeta_I(t)=\prod_{P\in
\cS_I}(1-t^{\pdeg(P,f_P^I)})^{-\chi(P)/\pdeg(P,f_P^I)}\\
\end{split}\]
\end{Theorem}
\begin{proof}
The proof is essentially  the same with  the proof of Theorem \ref{Varchenko}, 
given in the previous section.
We fix a toric modification 
$\pi: X\to \BC^n$
associated with a regular simplicial suddivision $\Si^*$ as in the holomorphic case.
 Take a strata $S= \widetilde V(\tau)^*$ or $S= \widetilde E(\tau)^{*}$
with $\tau=(P_1,\dots, P_k)$
and put  $\Xi=\cap_{i=1}^k \De(P_i;f)$. 
Let us take first  a toric chart $\si=(P_1,\dots, P_n)$ with $ \tau\prec\si$
with toric coordinates $\bfu=(u_{\si 1},\dots, u_{\si n})$.
For simplicity, we write $u_{\si j}=u_j$ hereafter.
The pull-back $\pi_\si^*f$ takes the following form.
Put $r_j=\rdeg_{P_j}f$ and $p_j=\pdeg_{P_j}f$ for $j=1,\dots, k$.
\begin{eqnarray}
 &\pi^*f=\prod_{j=1}^k u_{j}^{\frac{r_j+p_j}2}\bar u_{ j}^{\frac{r_j-p_j}2} \, \times  \widetilde f(\bfu,\bar \bfu)\\
& \widetilde f(\bfu)= \widetilde f_\De(\bfu',\bar\bfu')+ \widetilde R(\bfu,\bar\bfu)\label{reminder}
\end{eqnarray}
where $\bfu':=(u_{k+1},\dots,u_{n})$. 
The term $ \widetilde f_\De(\bfu',\bar\bfu')$ is a mixed polynomial which  does not contain the variables
$u_{ 1},\dots, u_{ k}$ by the strong polar weightedness assumption.
Namely we have
\[
 \pi_\si^*f_\De(\bfu,\bar\bfu)=\prod_{j=1}^k u_{ j}^{\frac{r_j+p_j}2}\bar u_{j}^{\frac{r_j-p_j}2} \, \times  \widetilde f_\De(\bfu',\bar\bfu').\]
 We will first see that $ \widetilde V(\tau)^*$  is  real analytically non-singular
and   $ \widetilde V$ is topologically non-singular on this stratum.
 First we assume that $R(\bfu,\bar\bfu)$ is a continuous  function such that 
the restriction of  $R$ to $\hat E(\tau)^*$ is zero for a while.
 Thus we see that  
$ \widetilde V(\tau)^*=
 \{({\bf 0},\bfu')\,|\,  \widetilde f_\De(\bfu',\bar\bfu')=0\}$.
For any $\bfx'\in  \widetilde V(\tau)^*$, 
put $\hat \bfx=(1,\dots, 1, \bfx')\in \BC_\si^{*n}$. 
Then $\hat \bfx \in \pi_\si^*f_\De\inv(0)= \widetilde f_\De\inv(0)$ and by the non-degenracy assumption of $f$ on the face $\De$,
$\hat \bfx$ is a non-singular point. That is, there exists a $j$, $k+1\le j\le n$, such that $\frac{\partial f_\De}{\partial u_j}(\hat\bfx)\ne 0$. This implies that 
 $ \widetilde V(\tau)^*$ is non-singular at $\bfx'$. Now we
 consider
$u_1,\dots, u_k$ as parameters and by Implicit function theorem, we can solve
 $ \widetilde f(\bfu)=0$ in $u_j$ so that 
$u_j$ is analytic in $\{u_i;i\ne j,k+1\le i\le n\}$ and continuous in
 $u_1,\dots, u_k$. This implies that $\tilde V$ is topological manifold.
This proves the assertions (1), (2).

Now we consider the Milnor fibration.
 The second term $ \widetilde R(\bfu,\bar\bfu)$  in (\ref{reminder}) is a
linear combination of monomials of the type
$u_{1}^{a_1}\bar u_{1}^{b_1}\cdots u_{n}^{a_n}\bar u_{n}^{b_n}$
with
\[\begin{split}
&a_i+b_i>0,\, i=1,\dots,k,\,\\
& a_j,b_j\ge 0,\,k+1\le j\le n.
\end{split}
\]
Here $a_i, b_i$ might be negative for $i\le k$.
See Example \ref{example}.
However the inequality $a_i+b_i>0$ is enough to see the continuity of $ \widetilde R(\bfu,\bar\bfu)$
and $\lim_{u_{\si i}\to 0} \widetilde R(\bfu,\bar\bfu)=0$ for any $1\le i\le k$.
See also the polar coordinate expression below for further detail.
Thus the function $ \widetilde R$ is a continuous  blow-analytic  function
in the sense of Kuo \cite{Kuo}.
We put $ \widetilde f_t= \widetilde f_\De+t  \widetilde R$ as before.
We use the notations:
\[\begin{split}
\hat f_t=\widetilde f_t\prod_{j=1}^k u_j^{\frac{r_j+p_j}2}\bar
 u_j^{\frac{r_j-p_j}2},\,\,
 \hat f_\De= \widetilde f_\De\prod_{j=1}^k u_j^{\frac{r_j+p_j}2}\bar
 u_j^{\frac{r_j-p_j}2}.
\end{split}
\]

 To show that the Milnor fibration
is well-defined  for any $0\le t\le 1$
by $\hat f_t$ over this stratum
and it is isomorphic to that of $\hat f_0=\hat f_\De$,
 we use {\em the polar toric coordinates}.
So put $u_j=\rho_j e^{i\theta_j}$ for $j=1,\dots, k$.
By the strong polar weighted homogenuity, 
the function takes the following form
\begin{eqnarray}
 &\pi_\si^*f(\rho,\theta,\bfu')= \rho_1^{r_1}\dots \rho_k^{r_k} e^{i p_1\theta_1}\cdots
e^{i p_k\theta_k} \widetilde f(\rho,\theta,\bfu')\\
& \widetilde f(\rho,\theta,\bfu')= \widetilde f_\De(\bfu')+ R(\rho, \theta,\bfu')\quad\qquad\qquad
\text{where}\\
&\pi_\si^*f_\De(\rho,\theta,\bfu') = \rho_1^{r_1}\dots \rho_k^{r_k}e^{i p_1\theta_1}\cdots 
e^{i p_k\theta_k} \widetilde f_\De(\bfu')
\end{eqnarray}
and  $\bfu'=(u_{k+1},\dots,u_{ n})$, $\rho=(\rho_1,\dots, \rho_k)$ and $\theta=(\theta_1,\dots, \theta_k)$.
We put 
\[
 \hat f_t= \rho_1^{r_1}\dots \rho_k^{r_k} e^{i p_1\theta_1}\cdots
e^{i p_k\theta_k} \widetilde f_t,\, \,\,\widetilde f_t:= \widetilde f_\De+t R.
\]
The reminder $R$ is an analytic function of the variables
$(\rho,\theta,\bfu')$ and contained in the ideal generated by 
$\rho_1,\dots,\rho_k$.
This implies that $R\equiv 0$ on $ \widetilde V(\tau)$.
The tubular neigborhood $ N(S)$ is defined by
\[
  N_\eps(S):\,\begin{cases} &\dist_{P_j}(\bfu)\le \eps, j=1,\dots, k,\,\quad
S= \widetilde
		E(\tau)^*\\
 &\dist_{P_j}(\bfu)\le \eps, j=1,\dots, k,\, | \widetilde f_\De|\le \sqrt{\eps},\,
S= \widetilde V(\tau)^*.
\end{cases}
\]
 The Milnor fiber $F_{t,\de}$ of $\hat f_t$  in this neighborhood is
defined  as:
\[\begin{split}
\rho_1^{r_1}\cdots \rho_k^{r_k}e^{i p_1\theta_1}\cdots 
e^{i p_k\theta_k}
 ( \widetilde f_\De(\bfu',\bar\bfu')+t \widetilde R
   (\rho,\theta,\bfu',\bar\bfu'))=\de
\end{split}
\]
where  $\rho=(\rho_1,\dots, \rho_k)$ and ${\theta}=(\theta_1,\dots,
 \theta_k)$, $\rho_j\le \eps,\,1\le j\le k$
and $\de\ll \eps$.
\subsubsection{Smoothness.}					
First,  we will show the smoothness of the Milnor fiber of $\hat f_t$.
Take any $\bfu_0'=(u_{0,k+1},\dots, u_{0,n})\in S$
and $\bfu_0\in F_{t,\de}$ with $\bfu_0=(u_{10},\dots,u_{k0},\bfu_0')$. By the
 non-degeneracy assumption and the strong-polar weighted homogenuity,
the jacobian matrix 
\[
 J_{>k} \widetilde f_\De:=\frac{\partial (\Re f_\De,\Im f_\De)}{\partial (x_{k+1},y_{k+1},\dots,
 x_n,y_n)}
\]
  has rank two.
Thus
\[
 J_{>k}(\hat f_t)=\rho_1^{r_1}\cdots \rho_k^{r_k}e^{i p_1\theta_1}\cdots 
e^{i p_k\theta_k}(J_{>k}( \widetilde f_\De)+t J_{>k}(R))(\bfu_0)
\]
and the second term of the right side is smaller than the first. Thus 
$J_{>k}(\hat f_t)$
has rank two over  an open neighborhood $U(\bfu_0')$ if $\de\ll\eps$.
As $S_\eps^{tr}$ is relatively  compact, we can cover by a finite such
 open sets.
%

\subsubsection{ Transversality.} We consider the transversality
of $F_{t,\de}$ and the boundary  $\partial N_{\eps}(S^{tr})$ at the intersection
of $B:=\partial N_{\eps}(S^{tr})\cap  N_{\eps}S'$
or $B'=\partial N_{\eps}(S^{tr})\cap  \widetilde V_{\eps}(S')$
where $S'= \widetilde E(\tau')^*$ with $\tau'=(P_1,\dots, P_k,\dots, P_m)$.
We use the canonical toric chart $\BC_\si^n$ of $\tau'$ and let
$\bfu=(u_1,\dots,u_n)$ be the coordinates for simplicity.
The boundaries are described as follows.
\[\begin{split}
B=&\{\bfu\,|\, \dist_{P_i}(\bfu)=\eps,\, k+1\le i\le m\}\quad
\text{or}\\
B'=&\{\bfu\,|\, \dist_{P_i}(\bfu)=\eps,\, k+1\le i\le m,\, | \widetilde f_{\De'}(\bfu'')|=\sqrt{\eps}\}.
\end{split}
\]
The argument is almost similar as that of the holomorphic case. So we
 consider the case $B'$.
Thus we
assume that $S'=V(\tau')$ is non-empty, {\it i.e.}, namely $\dim\,\De'\ge 1$.
Put $g=f_{\De'}(\bfu'')$
where $\bfu''=(u_{m+1},\dots, u_n)$. 
If $V(\tau')\ne \emptyset$, on a neighborhood of any point $\bfu_0''\in
 B'$, by the non-degeneracy of $f$ on $\De'$,
$g$ can be used as a member of an real analytic complex  coordinate chart.
 For example, we may assume that
there exists an open neighborhood $U(\bfu_0'')$ such that 
$(u_{ 1},\dots, u_{ m}, g, w_{m+2},\dots, w_{ n})$  is a real analytic complex coordinate chart on
$U(\bfu_0'')$. 
Here $( g, w_{m+2},\dots, w_{ n})$ is real analytic complex
 coordinates of $U(\bfu_0'')\cap \hat E(\tau')$.
 See the next subsection for the definition.
We use further the polar coordinates
\[\begin{split}
 &u_j=\rho_je^{i \theta_j},\,j=1,\dots, m,\,\,g=\rho_g e^{i\theta_g}\\
& \rho=(\rho_1,\dots, \rho_m),\theta=(\theta_1,\dots,\theta_m),\,
\bfw=(w_{m+2},\dots, w_{\si n}).
\end{split}
\]
The expression of (\ref{transverse}) is now written as follows.
\begin{eqnarray}
\begin{cases}
&\hat f_t(\bfu)=\left(\prod_{j=1}^m \rho_j^{r_j} e^{ip_j\theta_j}\right)
 \widetilde f_t(\rho,\theta,\rho_g,\theta_g,\bfw)\\
& \widetilde f_t(\bfu)=
\rho_g e^{i\theta_g}+R'(\rho_g,\theta_g,\bfw)+ t\bar R(\rho,\theta,\rho_g,\theta_g,\bfw),\,\\
&R=\left(\prod_{j=1}^m \rho_j^{r_j} e^{ ip_j\theta_j}\right)
 \bar R
\end{cases}
\end{eqnarray}
 Note that 
$R'$ and $\bar R$ are real analytic functions
of variables $\rho_g,\theta_g,\bfw$ and 
$\rho,\theta,\rho_g,\theta_g,\bfw$ respectively
and the restriction of $R'$.
Note that $R',\bar R$ are not analytic function in the variables
$u_1,\dots, u_m$. Here is the advantage of using polar coordinates.
Theoretically this is equivalent to consider the situation on the polar
 modification along $u_i=0,\,i=1,\dots,m$ in the sense of \cite{OkaMix}.

For simplicity, we assume that $d(P_1)=r_1>0$. Put $u_1=x_{\si1}+  i y_{\si1}$ as before.
Put
\begin{eqnarray*}
&a:=\frac{r_1+p_1}2,\quad b:=\frac{r_1-p_1}2\\
&U_{1,m}=\prod_{j=1}^m \rho_j^{r_j}e^{ip_j\theta_j}=u_1^{a}{\bar u_1}^{b}
\prod_{j=2}^m \rho_j^{r_j}e^{ip_j\theta_j}.
\end{eqnarray*}
We do the same discussion as in the case of holomorphic case.
Consider the Jacobian of $\hat f_t$ with respect to variables
$\{x_{\si1},y_{\si1},\rho_{k+1},\dots, \rho_m,\rho_g\}$. Recall that $\bar R$ is a Laurent series in the variable 
$u_1,\bar u_1$ but it only contains monomials $u_1^{m_1}{\bar u_1}^{n_1}$ with 
$m_1+n_1\ge 1$. Thus 
$|\frac{\partial \bar R}{  \partial x_{\si1}}|$ and $|\frac{\partial \bar R}{  \partial y_{\si1}}|$ is bounded from above when $|u_1|$ is small enough.
Thus as a complex vector,
\[
(\frac{\partial\hat f_t}{  \partial x_{\si1}},\frac{\partial\hat f_t}{  \partial y_{\si1}},
\frac{\partial\hat f_t}{\partial \rho_{k+1}},
\dots,\frac{\partial\hat f_t}{\partial \rho_g})\prosim
(\frac{a}{u_1}+\frac b{\bar u_1},i\frac{a}{u_1}-i\frac b{\bar u_1},\frac{r_{k+1}}{\rho_{k+1}},\dots,
\frac{r_m}{\rho_m},0)
\] 
Put $1/u_1=\al+\be i$.
Then the $2\times 2$ real matrix  $A$  of the first two coefficents
( as in the holomorphic case corresponding ) is
\[
A=\left(\begin{matrix}
\al(a+b)&-\be(a+b)\\
\be(a-b)&\al(a-b)
\end{matrix}
\right)
\]
and we see that $\det A=(a^2-b^2)(\al^2+\be^2)\ne 0$ as 
$a-b=p_1>0$, by the positive polar weightedness.
We consider the Jacobian matrix of $\{\Re\hat f_t,\Im \hat f_t, \dist_{P_k+1},\dots,
 \dist_{P_m},|g|\}$ in the variable $\{x_{\si1},y_{\si1},\rho_{k+1},\dots, \rho_m,\rho_g\}$.
Under projective equivalence for each gradient vector,
we get
\[
 \frac{\Re\partial \hat f_t,\Im\partial \hat f_t, \dist_{P_k+1},\dots, \dist_{P_m},|g|)}
{\partial(x_{\si1},y_{\si1},\rho_{k+1},\dots, \rho_m,\rho_g)}\prosim
 \left(
\begin{matrix}
A& *\\
0&I_{m-k+1}\\
\end{matrix}
\right)
\]
This matrix has rank $m-k+3$  as is expected.
Thus the transversality follows. 

The proof of Theorem \ref{main-result} is now given by the exact same
argument
as that of Theorem \ref{Varchenko}.
By Key Lemma below, the contribution to zeta function
is only from the strata $ \widetilde E(\tau)^{tr}$ where
 $\tau=(P_1,\dots, P_k)$ is a simplex of maximal face type.
\end{proof}
\subsubsection{Example}\label{example}
Consider the mixed function
\[
 f(\bfz,\bar \bfz)=z_1^3\bar z_1+ z_2^3\bar z_2+z_2^5.
\]
$f$
 is a non-degenerate mixed function of strongly polar weighted
 homogeneous type.
Then an ordinary blowing up $\pi: X\to \BC^2$
is the associated toric modification. Let us see in the chart
$\si=(P,E_2)$ with $P={}^t(1,1)$. Let $(u,v)$  be the toric chart.
Then we have
\[
 \pi_\si^*f(u,v)=u^3\bar u\left(v^3\bar v+1+\frac{u^2v^5}{\bar u}\right)
\]
and $R=u^2v^5/\bar u$. We see  $R$ is a continuous blow-analytic
function but not $C^1$ in $u$.
\subsection{Key lemma}
Let $f_1,\dots, f_n$ be complex valued real analytic functions
in an open set $U\subset \BC^n$.
We say that $(f_1,\dots, f_n)$ be  {\it  real analytic complex
coordinates}
if $(\Re f_1,\Im f_1,\dots, \Re f_n,\Im f_n)$ are real analytic
 coordinates of
$U$.
\begin{Lemma}\label{productMilnor}
Let $f(\bfg,\bar\bfg)=g_1^{r_1}{\bar g_1}^{p_1}\cdots g_\ell^{r_\ell}{\bar
 g_\ell}^{p_\ell}$.
Assume that $(g_1,\dots, g_\ell)$ are locally real analytic
 complex coordinates of $(\BC^{\ell},{\bf 0})$
and $r_j\ne p_j$ for each $j=1,\dots, \ell$.
Let $\hat q_j=r_j+p_j,\,q_j=r_j-p_j$ and put  $q_0=\gcd(q_1,\dots, q_\ell)$.
 Then 
 the Milnor fibration of $f$ exists at the origin and the Milnor
 fiber $F$  is homotopic to 
 $q_0$ disjoint copies of
 $(S^1)^{\ell-1}$ and the zeta function is given by
\[
 \zeta(t)=\begin{cases} (1-t^{q_1}),\quad & \ell=1\\
1,\quad & \ell\ge 2.\\
\end{cases}
\]
\end{Lemma}
\begin{proof}
Let $q_j=r_j-p_j$ and consider the linear $C^*$ action 
$(t,\bfg)\mapsto (tg_1,\dots, tg_\ell)$.
 Then $f$ can be understand as a polar
 homogeneous polynomial in $g_1,\dots, g_\ell$.
For the Milnor fibration, we can use the polydisk
$B_\eps'$
which is defined by
\[
 B_\eps':=\{(g_1,\dots, g_\ell)\,|\, |g_j|\le \eps,\,j=1,\dots, \ell\},
\]
as $\{B_\eps'; \eps>0\}$
is a  homotopy equivalent cofinal neighborhood system of the point 
${\bf 0}$.
Note that $B_\eps'$ is diffeomorphic to the usual complex polydisk
\[
\{(z_1,\dots, z_\ell)\in \BC^\ell\,|\, |z_j|\le \eps,\, j=1,\dots, \ell\}.
\]
 Using this action and polydisk $B_\eps'$,
the Milnor fibration can be identified with
\[
 f:B_\eps'\cap f\inv(D_\de^*)\to D_\de^*,\,\,
0\ne \de\ll \eps
\]
where 
$ D_\de=\{\rho\in \BC\,|\, |\rho|\le \eta\}$ and 
$ D_\de^*=D_\de\setminus\{0\}$.
Put $g_j=r_j  e^{i\theta_j},\,j=1,\dots, \ell$.
The Milnor fiber is given by
$ F=f\inv(\de)\cap B_\eps'=\amalg_{j=1}^{p_0}F_j$ (disjoint union)
and 
\[\begin{split}
F_j=&\{(g_1,\dots, g_\ell)\,|\,
r_1^{\hat q_1}\cdots r_\ell^{\hat q_\ell}=\de,\,
( \theta_1 q_1+\cdots+\theta_\ell
   q_\ell)/q_0=2\pi+
\frac{2j}{q_0}\pi\}\\
&\simeq (S^1)^{\ell-1},\,j=1,\dots, q_0.
\end{split}
\]
The monodromy map is given by the periodic map
\[
 h: F\to F,\quad
(g_1,\dots, g_\ell)\mapsto (g_1 e^{2\pi i/q_0},\dots, g_\ell e^{2\pi i/q_0})
\]
Thus we can see easily $F\simeq (S^1)^{{\ell-1}}$ and the zeta function is
 trivial
(Theorem 9.6, \cite{Milnor}).

\end{proof}

\subsection{An application and examples}
\subsubsection{Holomorphic principal part.}
Consider  a holomorphic function $g(\bfz)$ which is convenient and
non-degenerate.  Let $R(\bfz,\bar\bfz)$ be a mixed  analytic function 
such that $\Ga(R)$ is strictly  above $\Ga(g)$, {\em i.e.}  $\Ga(R)\subset \Int(\Ga_+(f))$.
Consider a  convenient non-degenerate mixed function
 $f(\bfz,\bar\bfz)  = g(\bfz)+R(\bfz,\bar\bfz)$. 
Then  the  Milnor fibration is determined by the principal part and therefore isomorphic to 
that of
$g_{\Ga}(\bfz)$. Thus the Milnor fibration  does not  change by adding high order mixed monomials above the Newton boundary. 
The proof can be given by showing the existence of uniform radius for
the Milnor fibration of $f_t:=g+t\,R$, using a Curve Selection Lemma
(\cite{Milnor, Hamm1}). We just copy the argument in \cite{OkaMix} for a
family
$f_t,\,0\le t\le 1$.
\subsubsection{Mixed covering.}

Consider a convenient non-degenerate mixed analytic  function $f(\bfz,\bar\bfz)$ of strongly polar weighted 
homogeneous face type.
Consider  a pair of positive integers $a>b\ge 0$ and consider the covering
mapping
\[
 \vphi:\quad \BC^n\to \BC^n,\quad
\vphi(\bfw)=(w_1^a\bar
w_1^b,\dots, w_n^a\bar w_n^b).
\] 
Put 
$g(\bfw,\bar\bfw)=f(\vphi(\bfw,\bar\bfw))=f(w_1^a\bar w_1^b,\dots, w_n^a\bar w_n^b)$.
Consider a face function with respect to $P$,
$f_P(\bfz,\bar \bfz)$ with $P=(p_1,\dots, p_n)$.
It is a strongly polar positive weighted homogeneous polynomial with  weight 
vector $P$.
Thus it is a linear combination of monomials
$\bfz^\nu{\bar\bfz}^\mu$ which satisfies the equalities:
\[
 \sum_{i=1}^n p_i(\nu_i+\mu_i)=\rdeg\,f \quad
\sum_{i=1}^n p_i(\nu_i-\mu_i)=\pdeg\, f.
\]
Then we can observe that
 $g_P(\bfw,\bar \bfw)=f_P(\vphi(\bfw))$ is a linear combination of mixed
 monomials $\bfw^{a\nu+b\mu}{\bar \bfw}^{a\mu+b\nu}$.
Thus  we have
\begin{eqnarray*}
& \rdeg_P \bfw^{a\nu+b\mu}{\bar \bfw}^{a\mu+b\nu}=(a+b)\rdeg\,f,\\
& \pdeg_P \bfw^{a\nu+b\mu}{\bar \bfw}^{a\mu+b\nu}=(a-b)\pdeg\,f
\end{eqnarray*}
which implies that $g_P$ is 
a stronly polar weighted homogeneous polynomial
of the same  weight $ P$ with radial degree  $\rdeg\,f\times (a+b)$ and 
polar degree $\pdeg\,f(\vphi(\bfw))=  d(P,f)(a-b)$.
$g_P$ is also non-degenerate 
as $\vphi:\BC^{*n}\to \BC^{*n}$ is an unbranched  covering  of degree $(a-b)^n$.
Therefore  the dual Newton diagram of $g$ is the same as that of $f$.

Let $\pi:\, X\to \BC^n$ be an admissible toric  modification. We use the same
notation as in the previous section.
For each $P\in \cS_I$, consider the mapping 
$\vphi^I:=\vphi|_{\BC^I}$
and its restriction to $F(g_P^I)$:
\[
 \begin{matrix}
 \BC^{*I}&\mapright{\vphi^I}& \BC^{*I}\\
\mapup{\iota}&&\mapup{\iota}\\
F^*(g_P^I)&\mapright{\vphi_P^I}&F^*(f_P^I)
\end{matrix}
\]
Here 
 the toric Milnor fibers are defined by:
\[\begin{split}
& F^*(g_P^I)=\{\bfw\in \BC^{*I}\,|\, g_P^I(\bfw^I,\bar\bfw^I)=1\}\\
& F^*(f_P^I)=\{\bfw\,|\, f_P^I(\bfz^I)=1\}.
\end{split}
\]
Put $\chi(P,f)=\chi(F^*(f_P^I))$ and $\chi(P,g)=\chi(F^*(g_P^I))$.
As $\vphi_P^I$ is a $(a-b)^{|I|}$-fold covering, we have
\begin{Proposition}\label{mixed-covering}
\begin{eqnarray}\label{Euler1}
 \chi(P,g)=(a-b)^{|I|}\chi(P,f),\qquad\qquad\\
\begin{cases} &\rdeg_P g=(a+b)\rdeg_P\,f\\
& \pdeg_P \, g_P=(a-b)\pdeg_P\,f_P
\end{cases}
\end{eqnarray}
\end{Proposition} 
From this observation, we have
\begin{Theorem}\label{covering}
 Let $g(\bfw,\bar \bfw)=f(\vphi(\bfw,\bar\bfw))$ and assume that 
$f(\bfz,\bar\bfz)$ is convenient non-degenerate mixed function
of strongly polar positive  weighted homogeneous face type.
Then $g(\bfw,\bar\bfw)=f(\vphi(\bfw,\bar\bfw))$
is a non-degenerate mixed function of strongly polar
positive weighted face type.
 The zeta functions of the Milnor fiberings of  $f$ and $g$ are  given by
\[\begin{split}
&\zeta_f(t)=\prod_I\zeta_{f,I}(t),\,\,
\zeta_{f,I}=\prod_{P\in \cS_I}(1-t^{\pdeg_P f_P})^{\chi(P,f)/\pdeg_P f_P}\\
& \zeta_g(t)=\prod_{I}\zeta_{g,I}(t),\,\,
\zeta_{g,I}(t)=\prod_{P\in \cS_I}(1-t^{\pdeg_P g_P})^{\chi(P,g)/\pdeg_P g_P}.
\end{split}
\]
Furthermore
$\zeta_g(t)$ is determined by that of $\zeta_f(t)$ using the above Proposition \ref{mixed-covering}.
\end{Theorem}
\subsubsection{Case $a-b=1$}.
We assume that $a-b=1$. Then Theorem \ref{covering} says that 
$\zeta_f(t)=\zeta_g(t).$
In this case, $\vphi:\BC^{*n}\to\BC^{*n}$ is a 
homeomorphism
which extends homeomorphically to $\BC^n\to \BC^n$.
This suggests the following.
\begin{Corollary} Assume $a-b=1$ in the situation of Theorem
 \ref{covering}.
\begin{enumerate}
\item The Milnor fibrations of $f(\bfz,\bar\bfz)$ and $g(\bfw,\bar \bfw)$ are homotopically equivalent
and the links are homotopic and   their zeta functions coincide. 
\item If in addition, $f=f_P$ is a strongly polar weighted homogeneous polynomial,
the Milnor fibrations $f,g$ are topologically equivalent
and 
the links are homeomorphic.
\end{enumerate}
\end{Corollary}
\begin{proof}

First, by Theorem 33 (\cite{OkaMix}), $f$ and $g$ have stable radius for
 their Milnor fibrations (in the first presentation). Take a common stable radius $r_0>0$.
For $r_0\ge r>  0$, put
\[
    \begin{split}
&B_{r}^*=B_{r}^{2n}\setminus \{{\bf 0}\}
=\{\bfz\in \BC^n\,|\, \|\bfz\|\le r,\,\bfz\ne {\bf 0}\}\\
&K_{r}(f)=f\inv(0)\cap S_r^{2n-1},\, K_{r}(g)=g\inv(0)\cap S_r^{2n-1}\\
&V_{r}^*(f)=f\inv(0)\cap B_{r}^*,\,\,
V_{r}^*(g)=g\inv(0)\cap B_{r}^*.
\end{split}
   \]
Then the following fibrations are obviously equivalent to the respective  Milnor fibrations:
\[
 \begin{cases} & f/|f|:\,\,B_{r}^*\setminus V_{r}^*(f) \to S^1\\
  & g/|g|:\,\,B_{r}^*\setminus V_{r}^*(g) \to S^1\\
  \end{cases}
\]
Our homeomorphism $\vphi$ preserves the values of $g$ and $f$.
 For any $r\le r_0$,
we can find a decerasing sequence of positive real numbers
 $r_{i}<r_0, i=1,2,\dots$ so that  $ \vphi(B_{r_{2i}}^*)\supset B_{r_{2i+1}}^*$ and
\[\begin{matrix}
 \vphi((B_{r_{2i}}^*, V_r^*(g)))&\supset&( B_{r_{2i+1}}^*, V_{r_{2i+1}}^*(f))&\supset & \vphi((B_{r_{2i+2}}^, V_{r_{2i+2}}^*(g)))\\
 \end{matrix}
\]
By the cofinal homotopy equivalence sequence argument,
\[
 \vphi:(B_{r_{2i}}^*, V_{r_{2i}}^*(g))\to
 (B_{r_{2i-1}}^*,V_{r_{2i-1},0}^*(f))
\]
 is a homotopy equivalence. 
 
 Assume now that  $f=f_P$ is a strongly polar weighted homogeneous polynomial.
Recall that we have $\BC^*$-action defined by
$t\circ \bfw=(w_1 t^{p_1},\dots, w_nt^{p_n})$ and 
$t\circ \bfz=(z_1 t^{p_1},\dots, z_nt^{p_n})$. Put
$d=\pdeg_P f=\pdeg_P g$.
The monodromy mapping of
 the global Milnor fibrations are given by
\[
 \begin{split}
h_g:F(g)\to F(g),\quad \bfw\mapsto e^{  i /d}\circ \bfw,\\
h_f:F(f)\to F(f),\quad \bfz\mapsto e^{  i /d}\circ \bfw
\end{split}
\]
where $F(g)=g\inv (1)\subset \BC^n$ and $F(f)=f\inv(1)\subset \BC^n$.
As is easily observed, we get the equality
$\vphi(h_g(\bfw))=h_f(\vphi(\bfw))$.
More precisely
the equivalence of the global fibration follows from the diagram:
\begin{eqnarray*}
\begin{matrix}
\BC^n\setminus g\inv(0)&\mapright{\vphi}&\BC^{n}\setminus f\inv(0)\\
\mapdown{g}&&\mapdown{f}\\
\BC^*&=&\BC^*\\
\end{matrix}
\end{eqnarray*}
and $\vphi$ is $S^1$-equivalent. Now we consider the links on the unit spheres
\[\begin{split}
 K_f=\{\bfz\in \BC^n\,|\, \|\bfz\|=1,\,f(\bfz,\bar\bfz)=0\},\\
K_g=\{\bfw\in \BC^n\,|\, \|\bfw\|=1,\,g(\bfw,\bar\bfw)=0\}.
\end{split}
\]
The mappining $\vphi$ does not keep the norm. So we need only normalize
 it.
This follows from the following observation for $\BR^+$-actions:
\[
 \vphi(t\circ \bfw)=t^{2b+1}\circ \vphi(\bfw),\, t\in \BR^+.
\]
This equality implies that $\vphi$ maps a $\BR^+$-orbit to a $\BR^+$-orbit.
The hypersurface $f\inv(0)$  and $g\inv(0)$ are  invariant under this action.
For any  non-zero $\bfz$, along  the orbit $\bfz(t):=t\circ \bfz$.
$\|\bfz(t)\|$ is monotone increasing for $0<t<\infty$.
Let us define the normalization map
$\psi:\BC^n\setminus \{{\bf 0}\}\to S^{2n-1}$ by
$\psi(\bfz)=\bfz(\tau)$ where $\tau$ is the unique positive real
 number
so that $\|\bfz(\tau)\|=1$.
We define the homeomorphism $ \widetilde \vphi: K_g\to K_f$ by
$ \widetilde \vphi(\bfw)=\psi(\vphi(\bfw))$. This proves the second  assertion.
\end{proof}
\subsubsection{Holomorphic case.}
The most interesting case is when $f(\bfz)$ is a non-degenerate
holomorphic functon. Obviously $f(\bfz)$ is of strongly polar weighted
homogeneous face type.
\begin{Theorem}
Let $f(\bfz)$ be a convenient nondegenerate holomorphic function and
 let $g(\bfw,\bar\bfw)=f(\vphi(\bfw,\bar\bfw))$,
$\vphi(\bfw,\bar\bfw)=(w_1^a\bar w_1^{b},\dots, w_n^a \bar w_n^{b})$.
Then $(\bfw,\bar\bfw)$ is a non-degenerate mixed function of
strongly polar positive weighted homogeneous face type.
\newline If $a-b=1$, the Milnor
 fibaration
of $g$ is homotopically equivalent to that of $f(\bfz)$.
\end{Theorem}
The mixed functions $g(\bfw,\bar\bfw)$ obtained  from non-degenerate
  holomorphic functions through the pull-back by a mixed covering $\vphi$
  give many intersting 
examples of non-degenerate mixed functions of strongly polar
weighted homogeneous  face type.
\subsection{Examples}
(1) Consider the following strongly  polar homogeneous polynomial (Example
59, \cite{OkaMix})
\[
 g_t(\bfw,\bar\bfw)=-2 w_1^2\bar w_1+w_2^2\bar w_2+t w_1^2\bar w_2.
\]
For $t=0$, $g_0$ is pull-back of $f(\bfz)=-2 z_1+z_2$ by the 
mixed covering mapping
$\vphi:\BC^2\to \BC^2$ with $\vphi(\bfw)=(w_1^2\bar w_1,w_2^2\bar w_2)$.
Thus $g_0$ is equivalent to the trivial knot $f=0$ 
and $F$ is contaractible and $\zeta(t)=(1-t)$ by Theorem \ref{covering}.
On the other hand, for $t>1$, we know that $g_t$ is non-degenerate and 
$\chi(F^*)=-3,\,\chi(F)=-1$ and $\zeta(t)=(1-t)^{-1}$.
This shows that there are mixed functions
of strongly polar weighted homogeneous face type
 which are not the pull-back of 
holomorphic functions.

(2) Consider the moduli space $\cM(P, pab,(p+2r)ab)$ of convenient
non-degenerate
strongly polar weighted
homogeneous polynomials $f(\bfz,\bar\bfz)$ of two variables $\bfz=(z_1,z_2)$
with weight $P=(a,b),\,\gcd(a,b)=1$ and 
$\pdeg_Pf=pab,\,\rdeg_P f=(p+2r)ab$. Here $r$ is a positive integer.
Put $h_s(\bfz,\bar\bfz)=|z_1|^{2s}+|z_2|^{2s}$.
Consider the polynomials
\[
 f_s(\bfz,\bar \bfz)=h_{r-s}(\bfz,\bar\bfz)\prod_{j=1}^p(z_1^b-\al_jz_2^a)\prod_{k=1}^s
(z_1^b-\be_k z_2^a)(\bar z_1^b-\bar\ga_k\bar z_2^a)
\]
where $0\le s\le r$ and 
$\al_1,\dots,\al_p,\be_1,\ga_1,\dots, \be_s,\ga_s$ are mutually
distinct non-zero complex numbers.
As $V(f_s)\subset \BP^1$ consists of $p+2s$ points,
we have $\chi(F^*)=-p(p+2s)$ and $\zeta_{f_s}(t)=(1-t^p)^{p+2s-2}$ (See \cite{MC}).

For any $n=2+m$, we can consider a join type strongly polar weighted
homogeneous polynomial of $m+2$ bariables $\bfz,\bfw$ with $\bfw\in \BC^m$:
\[
 F_s(\bfz,\bfw,\bar\bfz,\bar\bfw)=f_s(\bfz,\bar\bfz)+w_1^{p+r}\bar
 w_1^{r}+
\dots+w_n^{p+r}\bar w_n^r.
\]
Then by join theorem (\cite{Molina}), the Milnor fiber of
$F_s,\,s=0,\dots, r$ have different topology.
In fact, the Milnor fiber is homotopic to a bouquet of spheres and 
the Milnor number is $(p-1)^m (p^2+(2s-2)p+1).$ Thus topology of mixed
polynomials is  not combinatorial invariant.

\def\cprime{$'$} \def\cprime{$'$} \def\cprime{$'$} \def\cprime{$'$}
  \def\cprime{$'$} \def\cprime{$'$} \def\cprime{$'$} \def\cprime{$'$}

\end{document}